\documentclass[12pt,a4paper]{article}

\usepackage{amssymb}
\usepackage{amsmath,amsthm}
\usepackage{verbatim}
\usepackage{graphicx}
\newtheorem{theorem}{Theorem}
\newtheorem*{kcaserepeat}{Theorem 5}
\newtheorem*{2caserepeat}{Theorem 4}

\newtheorem{proposition}[theorem]{Proposition}
\newtheorem{conjecture}[theorem]{Conjecture}
\newtheorem{lemma}[theorem]{Lemma}
\newtheorem{corollary}[theorem]{Corollary}
\newtheorem{claim}{Claim}
\theoremstyle{remark}

\begin{document}
\title{Generating all subsets of a finite set with disjoint
unions}
\date{May 2011}

\author{David Ellis\thanks{St John's College, Cambridge, CB2 1TP, UK.},
Benny Sudakov\thanks{Department of Mathematics,
UCLA, Los Angeles, CA 90095, USA. Research supported in part by NSF CAREER award
DMS-0812005 and by a USA-Israeli BSF grant.}}
\maketitle
\begin{abstract}
If \(X\) is an \(n\)-element set, we call a family
\(\mathcal{G} \subset \mathcal{P}X\) a \(k\)-\textit{generator for} \(X\) if
every \(x \subset X\) can be expressed as a union of at most
\(k\) disjoint sets in \(\mathcal{G}\). Frein, L\'ev\^eque and Seb\H o
\cite{Leveque} conjectured that for \(n > 2k\), the smallest \(k\)-generators
for \(X\) are obtained by taking a partition of \(X\)
into classes of sizes as equal as possible, and taking the union of the
power-sets of the classes. We prove this conjecture for all sufficiently large
\(n\) when \(k = 2\), and for \(n\) a sufficiently
large multiple of \(k\) when \(k \geq 3\).
\end{abstract}

\textit{Keywords:} generator, disjoint unions.

\textit{2000 MSC:} 05Dxx.

\section{Introduction}
Let \(X\) be an \(n\)-element set, and let $\mathcal{P}X$ denote the set of all
subsets of $X$. We call a family \(\mathcal{G} \subset \mathcal{P}X\) a
\(k\)-\textit{generator for} \(X\) if every \(x \subset
X\) can be expressed as a union of at most \(k\) disjoint
sets in \(\mathcal{G}\). For example, let \((V_{i})_{i=1}^{k}\) be a partition
of \(X\) into \(k\) classes of sizes as equal as possible; then
\[\mathcal{F}_{n,k} := \bigcup_{i=1}^{k} \mathcal{P}(V_{i}) \setminus
\{\emptyset\}\]
is a \(k\)-generator for \(X\). We call a \(k\)-generator of this form {\em
canonical}. If \(n = qk+r\), where \(0 \leq r <k\), then
\[|\mathcal{F}_{n,k}| = (k-r)(2^q-1)+r(2^{q+1}-1) = (k+r)2^{q}-k.\]

Frein, L\'ev\^eque and Seb\H o \cite{Leveque} conjectured that for any \(k \leq
n\), this is the smallest possible size of a \(k\)-generator for \(X\).

\begin{conjecture}[Frein, L\'ev\^eque, Seb\H o]
\label{conj:frein}
If \(X\) is an \(n\)-element set, \(k \leq n\), and \(\mathcal{G} \subset
\mathcal{P}X\) is a \(k\)-generator for \(X\), then \(|\mathcal{G}| \geq
| \mathcal{F}_{n,k}|\). If \(n > 2k\), equality holds only if \(\mathcal{G}\) is
a canonical \(k\)-generator for \(X\).
\end{conjecture}
\noindent
They proved this for \(k \leq n \leq 3k\), but their methods do not seem to
work for larger \(n\).

For \(k=2\), Conjecture \ref{conj:frein} is a weakening of a conjecture of
Erd\H os. We call a family \(\mathcal{G} \subset \mathcal{P}X\) a
\(k\)-\textit{base for} \(X\) if every \(x \subset X\) can be
expressed as a union of at most \(k\) (not necessarily disjoint) sets in
\(\mathcal{G}\). Erd\H os (see \cite{furedi}) made the following

\begin{conjecture}[Erd\H os]
\label{conj:erdos}
If \(X\) is an \(n\)-element set, and \(\mathcal{G} \subset \mathcal{P}X\) is a
\(2\)-base for \(X\), then \(|\mathcal{G}| \geq |\mathcal{F}_{n,2}|\).
\end{conjecture}
\noindent
In fact, Frein, L\'ev\^eque and Seb\H o \cite{Leveque} made the analogous
conjecture for all \(k\).

\begin{conjecture}[Frein, L\'ev\^eque, Seb\H o]
\label{conj:erdosgen}
If \(X\) is an \(n\)-element set, \(k \leq n\), and \(\mathcal{G} \subset
\mathcal{P}X\) is a \(k\)-base for \(X\), then \(|\mathcal{G}| \geq
| \mathcal{F}_{n,k}|\). If \(n > 2k\), equality holds only if \(\mathcal{G}\) is
a canonical \(k\)-generator for \(X\).
\end{conjecture}
\noindent
Again, they were able to prove this for \(k \leq n \leq 3k\).

In this paper, we study \(k\)-generators when \(n\) is large compared
to \(k\). Our main results are as follows.

\begin{theorem}
\label{thm:2case}
If \(n\) is sufficiently large, \(X\) is an \(n\)-element set, and
\(\mathcal{G} \subset \mathcal{P}X\) is a 2-generator for \(X\), then
\(|\mathcal{G}| \geq |\mathcal{F}_{n,2}|\). Equality holds only if
\(\mathcal{G}\) is of the form \(\mathcal{F}_{n,2}\).
\end{theorem}

\begin{theorem}
\label{thm:kcase}
If \(k \in \mathbb{N}\), \(n\) is a sufficiently large multiple of \(k\), \(X\)
is an \(n\)-element set, and \(\mathcal{G}\) is a \(k\)-generator for \(X\),
then \(|\mathcal{G}| \geq |\mathcal{F}_{n,k}|\). Equality holds only if
\(\mathcal{G}\) is of the form \(\mathcal{F}_{n,k}\).
\end{theorem}

In other words, we prove Conjecture \ref{conj:frein} for all sufficiently large \(n\)
when \(k = 2\), and for \(n\) a sufficiently
large multiple of \(k\) when \(k \geq 3\). We use some ideas of Alon and Frankl
\cite{alon}, and also techniques of the first author from \cite{ellis}, in
which
asymptotic results were obtained.

As noted in \cite{Leveque}, if \(\mathcal{G} \subset \mathcal{P}X\) is a
\(k\)-generator (or even a \(k\)-base) for \(X\), then the number of ways of
choosing at most \(k\) sets from \(\mathcal{G}\) is
clearly at least the number of subsets of \(X\). Therefore $|\mathcal{G}|^k
\geq 2^n$, which immediately gives
$$|\mathcal{G}| \geq 2^{n/k}.$$
Moreover, if \(|\mathcal{G}|=m\), then
\begin{equation}
\label{eq:countingbound}
\sum_{i=0}^{k} {m \choose i} \geq 2^{n}.
\end{equation}
Crudely, we have
\[\sum_{i=0}^{k-1} {m \choose i} \leq 2m^{k-1},\]
so
\[\sum_{i=0}^{k} {m \choose i} \leq {m \choose k} + 2m^{k-1}.\]
Hence, if \(k\) is fixed, then
\[(1+O(1/m)) {m \choose k} \geq 2^n,\]
so
\begin{equation}
\label{eq:trivbound}
 |\mathcal{G}| \geq (k!)^{1/k}2^{n/k}(1-o(1)).
\end{equation}

Observe that if \(n = qk+r\), where \(0 \leq r <k\), then
\begin{equation}
 \label{eq:crudebound}
| \mathcal{F}_{n,k}| = (k+r)2^{q}-k < (k+r)2^{q} = k2^{n/k}(1+r/k)2^{-r/k} <
c_{0}k2^{n/k},
\end{equation}
where
\[c_{0} := \frac{2}{2^{1/\log 2}\log 2} = 1.061 \textrm{ (to 3 d.p.)}.\]

Now for some preliminaries. We use the following standard notation. For \(n \in
\mathbb{N}\), \([n]\) will denote the
set \(\{1,2,\ldots,n\}\). If \(x\) and \(y\) are disjoint sets, we will sometimes write their union as \(x \sqcup y\), rather than \(x \cup y\), to emphasize the fact that the sets are disjoint.

If \(k \in \mathbb{N}\), and \(G\) is a graph,
\(K_{k}(G)\) will denote the number of \(k\)-cliques
in \(G\). Let \(T_{s}(n)\) denote the \(s\)-partite Tur\'an graph (the complete
$s$-partite graph on $n$ vertices with parts of sizes as equal as possible),
and let
\(t_{s}(n) = e(T_{s}(n))\). For \(l \in \mathbb{N}\), \(C_{l}\) will denote the
cycle of length \(l\).

If \(F\) is a (labelled) graph on \(f\) vertices, with vertex-set
\(\{v_1,\dots,v_f\}\) say, and \(\mathbf{t} = (t_1,\dots,t_f) \in
\mathbb{N}^{f}\), we define the \(\mathbf{t}\)-{\em blow-up} of \(F\),
\(F \otimes \mathbf{t}\), to be the graph obtained by replacing \(v_i\) with an
independent set \(V_i\) of size \(t_i\), and joining each vertex of \(V_{i}\)
to each vertex of \(V_{j}\) whenever \(v_i
v_j\) is an edge of \(F\). With slight abuse of notation, we will write \(F
\otimes t\) for the symmetric blow-up \(F \otimes (t,\dots,t)\).

If \(F\) and \(G\) are graphs, we write \(c_{F}(G)\) for the number of
injective graph homomorphisms from \(F\) to \(G\), meaning injections from
\(V(F)\) to \(V(G)\) which take edges of \(F\) to edges of
\(G\). The {\em density of F in G} is defined to be \[d_{F}(G) =
\frac{c_{F}(G)}{|G|(|G|-1)\cdots(|G|-|F|+1)},\] i.e. the probability that a
uniform random injective map from \(V(F)\) to \(V(G)\) is a
graph homomorphism from \(F\) to \(G\). Hence, when \(F = K_{k}\), the density
of \(K_{k}\)'s in an \(n\)-vertex graph \(G\) is simply \(K_{k}(G)/{n \choose
k}\).

Although we will be interested in the density \(d_{F}(G)\), it will sometimes
be more convenient to work with the following closely related quantity, which
behaves very nicely when we take blow-ups. We
write \(\textrm{Hom}_{F}(G)\) for the number of homomorphisms from \(F\) to
\(G\), and we define the {\em homomorphism density of F in G} to be \[h_{F}(G)
= \frac{\textrm{Hom}_{F}(G)}{|G|^{|F|}},\] i.e.
the probability that a uniform random map from \(V(F)\) to \(V(G)\) is a graph
homomorphism from \(F\) to \(G\).

Observe that if \(F\) is a graph on \(f\) vertices, and \(G\) is a graph on
\(n\) vertices, then the number of homomorphisms from \(F\) to \(G\) which are
not injections is clearly at most
\[{f \choose 2} n^{f-1}.\]
Hence,
\begin{equation}
\label{eq:homodens}
d_{G}(F) \geq \frac{h_{G}(F)n^{f} - {f \choose 2}n^{f-1}}{n(n-1)\cdots(n-f+1)} \geq h_{G}(F)-O(1/n),
\end{equation}
if \(f\) is fixed. In the other direction,
\begin{equation}
\label{eq:denshomo}
d_{F}(G) \leq \frac{n^{f}}{n(n-1)\cdots(n-f+1)}h_{F}(G) \leq (1+O(1/n))
h_{F}(G)
\end{equation}
if \(f\) is fixed. Hence, when working inside large graphs, we can pass freely
between the density of a fixed graph \(F\) and its homomorphism density, with
an `error' of only \(O(1/n)\).

Finally, we will make frequent use of the AM/GM inequality:
\begin{theorem}
If \(x_1,\ldots,x_n \geq 0\), then
\[\left(\prod_{i=1}^{n} x_i\right)^{1/n} \leq \frac{1}{n}\sum_{i=1}^{n}x_i.\]
\end{theorem}

\section{The case \(k \mid n\) via extremal graph theory.}

For \(n\) a sufficiently large multiple of \(k\), it turns out to be possible
to prove Conjecture \ref{conj:frein} using stability versions of Tur\'an-type
results. We will prove the following

\begin{kcaserepeat}
If \(k \in \mathbb{N}\), \(n\) is a sufficiently large multiple of \(k\), \(X\)
is an \(n\)-element set, and \(\mathcal{G}\) is a \(k\)-generator for \(X\),
then \(|\mathcal{G}| \geq |\mathcal{F}_{n,k}|\). Equality holds only if
\(\mathcal{G}\) is of the form \(\mathcal{F}_{n,k}\).
\end{kcaserepeat}

We need a few more definitions. Let \(H\) denote the graph with vertex-set
\(\mathcal{P}X\), where we join two subsets \(x,y \subset X\) if they are
disjoint. With slight abuse of terminology, we call
\(H\) the `Kneser' graph on \(\mathcal{P}X\) (although this usually means the
analogous graph on \(X^{(r)}\)). If \(\mathcal{F},\mathcal{G} \subset
\mathcal{P}X\), we say that \(\mathcal{G}\)
\(k\)-\textit{generates} \(\mathcal{F}\) if every set in \(\mathcal{F}\) is a
disjoint union of at most \(k\) sets in \(\mathcal{G}\).

\vspace{0.2cm}
\noindent
{\bf The main steps of the proof:}\, First, we will show that for any
\(\mathcal{A} \subset \mathcal{P}X\) with \(|\mathcal{A}| \geq
\Omega(2^{n/k})\), the density of \(K_{k+1}\)'s in the induced subgraph
\(H[\mathcal{A}]\) is \(o(1)\).

Secondly, we will observe that if \(n\) is a sufficiently large multiple of
\(k\), and \(\mathcal{G} \subset \mathcal{P}X\) has size close to
\(|\mathcal{F}_{n,k}|\) and \(k\)-generates almost all subsets
of \(X\), then \(K_{k}(H[\mathcal{G}])\) is very close to
\(K_{k}(T_{k}(|\mathcal{G}|))\), the number of \(K_{k}\)'s in the \(k\)-partite
Tur\'an graph on $|\mathcal{G}|$ vertices.

We will then prove that if \(G\) is any graph with small \(K_{k+1}\)-density,
and with \(K_{k}(G)\) close to \(K_{k}(T_{k}(|G|))\), then \(G\) can be made
\(k\)-partite by removing a small number of edges.
This can be seen as a (strengthened) variant of the Simonovits Stability
Theorem \cite{simonovits}, which states that any \(K_{k+1}\)-free graph \(G\)
with \(e(G)\) close to the maximum \(e(T_{k}(|G|))\),
can be made \(k\)-partite by removing a small number of edges.

This will enable us to conclude that \(H[\mathcal{G}]\) can be made
\(k\)-partite by the removal of a small number of edges, and therefore the structure of \(H[\mathcal{G}]\) is close to that of the Tur\'an graph
\(T_{k}(|\mathcal{G}|)\). This in turn will enable us to show that
the structure of \(\mathcal{G}\) is close to that of a canonical \(k\)-generator
\(\mathcal{F}_{n,k}\) (Proposition \ref{prop:kstab}).

Finally, we will use a perturbation argument to show that if \(n\) is
sufficiently large, and \(|\mathcal{G}| \leq |\mathcal{F}_{n,k}|\), then
\(\mathcal{G} = \mathcal{F}_{n,k}\), completing the proof.

\vspace{0.2cm}
In fact, we will first show that if \(\mathcal{A} \subset \mathcal{P}X\) with
\(|\mathcal{A}| \geq \Omega(2^{n/k})\), then the homomorphism density of
\(K_{k+1} \otimes t\) in \(H[\mathcal{A}]\) is
\(o(1)\), provided \(t\) is sufficiently large depending on \(k\). Hence, we
will need the following (relatively well-known) lemma relating the homomorphism
density of a graph to that of its blow-up.

\begin{lemma}
\label{lemma:blowupcount}
Let \(F\) be a graph on \(f\) vertices, let \(\mathbf{t} = (t_1,t_2,\dots,t_f)
\in \mathbb{N}^f\), and let \(F \otimes \mathbf{t}\) denote the
\(\mathbf{t}\)-blow-up of \(F\). If the homomorphism density of \(F\) in \(G\)
is \(p\), then the homomorphism density of \(F \otimes
\mathbf{t}\) in \(G\) is at least \(p^{t_1t_2\cdots t_f}\).
\end{lemma}
\begin{proof}

This is a simple convexity argument, essentially that of \cite{simonovits}. It
will suffice to prove the statement of the lemma when \(\mathbf{t} =
(1,\dots,1,r)\) for some \(r \in \mathbb{N}\). We think of \(F\) as a
(labelled)
graph on vertex set \([f] = \{1,2,\dots,f\}\), and \(G\) as a (labelled) graph
on vertex set \([n]\). Define the function \(\chi: [n]^{f} \to \{0,1\}\) by
\[\chi(v_1,\dots,v_f) = \left\{\begin{array}{ll} 1 & \textrm{if } i \mapsto
v_i \textrm{ is a homomorphism from }F \textrm{ to } G, \\
                                                           0 &
\textrm{otherwise.}
                                                          \end{array}\right.\]

Then we have
\[h_{F}(G) = \frac{1}{n^f}
\sum_{(v_1,\dots,v_f) \in
[n]^{f}}
\chi(v_1,\dots,v_f) = p.\]
The homomorphism density \(h_{F \otimes (1,\dots,1,r)}(G)\) of \(F
\otimes (1,\dots,1,r)\) in \(G\) is:

\begin{align*}
\label{eq:convexity}
h_{F \otimes
(1,\dots,1,r)}(G)&=\frac{1}{n^{f-1+r}}\sum_{(v_1,\dots,v_{f-1},v_f^{(1)},v_f^{(2)},\dots,v_f^{(r)})
\in
[n]^{f-1+r}} \prod_{i=1}^r
\chi(v_{1},\dots,v_{f-1},v_f^{(i)})\\
& = \frac{1}{n^{f-1}} \sum_{(v_1,\dots,v_{f-1}) \in [n]^{f-1}}
\left(\frac{1}{n}\sum_{v_f \in [n]} \chi(v_1,\dots,v_{f-1},v_f)\right)^r
\\
& \geq \left(\frac{1}{n^{f-1}} \sum_{(v_1,\dots,v_{f-1}) \in [n]^{f-1}}
\left(\frac{1}{n}\sum_{v_f \in [n]}
\chi(v_1,\dots,v_{f-1},v_f)\right)\right)^r \\
& = \left(\frac{1}{n^{f}} \sum_{(v_1,\dots,v_{f-1},v_f) \in [n]^{f}}
\chi(v_1,\dots,v_{f-1},v_f)\right)^r \\
& = p^r.
\end{align*}
Here, the inequality follows from applying Jensen's Inequality to the
convex function \(x \mapsto x^r\). This proves the lemma for \(\mathbf{t} =
(1,\ldots,1,r)\). By symmetry, the statement of the lemma holds for all vectors
of the
form \((1,\ldots,1,r,1,\ldots,1)\). Clearly, we may obtain \(F \otimes
\mathbf{t}\)
from \(F\) by a sequence of blow-ups by
these vectors, proving the lemma.
\end{proof}

The following lemma (a rephrasing of Lemma 4.2 in Alon and Frankl \cite{alon})
gives an upper bound on the homomorphism density of \(K_{k+1} \otimes t\) in
large induced subgraphs of the Kneser graph \(H\).

\begin{lemma}
\label{lemma:homodens}
If \(\mathcal{A} \subset \mathcal{P}X\) with \(|\mathcal{A}| = m =
2^{(\delta+1/(k+1))n}\), then
\[h_{K_{k+1} \otimes t}(H[\mathcal{A}]) \leq (k+1)2^{-n(\delta t-1)}.\]
\end{lemma}
\begin{proof}
We follow the proof of Alon and Frankl cited above. Choose \((k+1)t\) members
of \(\mathcal{A}\) uniformly at random with replacement, \((A_{i}^{(j)})_{1
\leq i \leq k+1,\ 1 \leq j \leq t}\). The homomorphism density of \(K_{k+1}
\otimes t\) in \(H[\mathcal{A}]\) is precisely the probability that the unions
\[U_{i} = \bigcup_{j=1}^{t}A_{i}^{(j)}\]
are pairwise disjoint. If this event occurs, then \(|U_{i}| \leq n/(k+1)\) for
some \(i\). For each \(i \in [k]\), we have
\begin{align*}
\textrm{Pr}\{|U_{i}| \leq n/(k+1)\} & = \textrm{Pr}\left(\bigcup_{S \subset
X:|S| \leq n/(k+1)} \left(\bigcap_{j=1}^{t} \{A_{i}^{(j)} \subset
S\}\right)\right)\\
& \leq \sum_{|S| \leq n/(k+1)} \textrm{Pr}\left(\bigcap_{j=1}^{t}
\{A_{i}^{(j)} \subset S\}\right)\\
& = \sum_{|S| \leq n/(k+1)} (2^{|S|}/m)^{t}\\
& \leq 2^{n} (2^{n/(k+1)}/m)^{t} \\
& = 2^{-n(\delta t-1)}.
\end{align*}
Hence,
\[\textrm{Pr}\left(\bigcup_{i=1}^{k} \{|U_{i}| \leq n/(k+1)\}\right) \leq
\sum_{i=1}^{k} \textrm{Pr}\{|U_{i}| \leq n/(k+1)\} \leq (k+1)2^{-n(\delta
t-1)}.\]
Therefore,
\[h_{K_{k+1} \otimes t}(H[\mathcal{A}]) \leq (k+1)2^{-n(\delta t-1)},\]
as required.
\end{proof}

From the trivial bound above, any \(k\)-generator \(\mathcal{G}\) has
\(|\mathcal{G}| \geq 2^{n/k}\), so \(\delta \geq 1/(k(k+1))\), and therefore,
choosing \(t=t_{k}:=2k(k+1)\), we see that
\[h_{K_{k+1} \otimes t_{k}}(H[\mathcal{G}]) \leq (k+1)2^{-n}.\]
Hence, by Lemma \ref{lemma:blowupcount},
\[h_{K_{k+1}}(H[\mathcal{G}]) \leq O_k\big(2^{-n/t_{k}^k}\big).\]
Therefore, by (\ref{eq:denshomo}),
\begin{equation}
\label{eq1}
d_{K_{k+1}}(H[\mathcal{G}]) \leq O_k\big(2^{-n/t_{k}^k}\big) \leq 2^{-a_{k}n}
\end{equation}
provided \(n\) is sufficiently large depending on \(k\), where \(a_{k} > 0\)
depends only on \(k\).

Assume now that \(n\) is a multiple of \(k\), so that \(|\mathcal{F}_{n,k}| =
k2^{n/k}-k\). We will prove the following `stability' result.

\begin{proposition}
\label{prop:kstab}
Let \(k \in \mathbb{N}\) be fixed. If \(n\) is a multiple of \(k\), and \(\mathcal{G}
\subset \mathcal{P}X\) has \(|\mathcal{G}| \leq (1+\eta)|\mathcal{F}_{n,k}|\)
and \(k\)-generates at least \((1-\epsilon)2^{n}\) subsets of \(X\), then there
exists an equipartition \((S_i)_{i=1}^{k}\) of \(X\) such that
\[|\mathcal{G} \cap \left(\cup_{i=1}^{k} \mathcal{P}S_i\right)| \geq
(1-C_{k}\epsilon^{1/k} - D_{k}
\eta^{1/k} - 2^{-\xi_{k} n})|\mathcal{F}_{n,k}|,\]
where \(C_{k},D_{k},\xi_{k} > 0\) depend only on \(k\).
\end{proposition}

We first collect some results used in the proof. We will need the following theorem of Erd\H{o}s \cite{erdos}.

\begin{theorem}[Erd\H{o}s]
\label{thm:erd}
If \(r \leq k\), and \(G\) is a \(K_{k+1}\)-free graph on \(n\) vertices, then
\[K_{r}(G) \leq K_{r}(T_{k}(n)).\]
\end{theorem}

We will also need the following well-known lemma, which states that a dense \(k\)-partite graph has an induced subgraph with high minimum degree.

\begin{lemma}
\label{lemma:highmindegree}
Let \(G\) be an \(n\)-vertex, \(k\)-partite graph with
\[e(G) \geq (1-1/k-\delta)n^2/2.\]
Then there exists an induced subgraph \(G' \subset G\) with \(|G'| = n' \geq
(1-\sqrt{\delta}) n\) and
minimum degree \(\delta(G') \geq (1-1/k-\sqrt{\delta})(n'-1)\).
\end{lemma}
\begin{proof}
We perform the following algorithm to produce \(G'\). Let \(G_{1} = G\).
Suppose that at stage \(i\), we have a graph \(G_{i}\) on \(n-i+1\) vertices.
If there is a vertex \(v\) of \(G_{i}\) with \(d(v) < (1-1/k-\eta)(n-i)\), let
\(G_{i+1} = G_{i} - v\); otherwise, stop and set \(G' = G_{i}\). Suppose the
process terminates after \(j = \alpha n\) steps. Then we have removed at most
\[(1-1/k-\eta) \sum_{i=1}^{j} (n-i) = (1-1/k-\eta)\left({n \choose 2} - {n-j
\choose 2}\right)\]
edges, and the remaining graph has at most
\[{k \choose 2}\left(\frac{n-j}{k}\right)^2 = (1-\alpha)^2(1-1/k)n^2/2\]
edges. But our original graph had at least
\[(1-1/k-\delta)n^2/2\]
edges, and therefore
\[(1-1/k-\eta)(1-(1-\alpha)^{2})n^2/2 + (1-\alpha)^{2} (1-1/k)n^2/2 \geq
(1-1/k-\delta)n^2/2,\]
so
\[\eta(1-\alpha)^{2} \geq \eta - \delta.\]
Choosing \(\eta = \sqrt{\delta}\), we obtain
\[\eta(1-\alpha)^{2} \geq \eta(1-\eta),\]
and therefore
\[(1-\alpha)^{2} \geq 1-\eta,\]
so
\[\alpha \leq 1-(1-\eta)^{1/2} \leq \eta.\]
Hence, our induced subgraph \(G'\) has order
\[|G'| = n' \geq (1-\sqrt{\delta})n,\]
and minimum degree
\[\delta(G') \geq (1-1/k-\sqrt{\delta})(n'-1).\]
\end{proof}

We will also need Shearer's Entropy Lemma.

\begin{lemma}[Shearer's Entropy Lemma, \cite{cgfs}]
\label{thm:shearer}
Let \(S\) be a finite set, and let \(\mathcal{A}\) be an {\em \(r\)-cover} of \(S\), meaning a collection of subsets of \(S\) such that every element of
\(S\) is contained in at least \(r\) sets in \(\mathcal{A}\). Let
\(\mathcal{F}\) be a collection of subsets of \(S\). For \(A \subset S\), let
\(\mathcal{F}_{A} = \{F \cap A:\ F \in \mathcal{F}\}\) denote the {\em
projection} of \(\mathcal{F}\) onto the set \(A\). Then
\[|\mathcal{F}|^{r} \leq \prod_{A \in \mathcal{A}}|\mathcal{F}_{A}|.\]
\end{lemma}

In addition, we require two `stability' versions of Tur\'an-type results in extremal graph theory. The first states that a graph with a very small \(K_{k+1}\)-density cannot have \(K_{r}\)-density much higher than the \(k\)-partite Tur\'an graph on the same number of vertices, for any \(r \leq k\).

\begin{lemma}
\label{lemma:upperbound}
Let \(r \leq k\) be integers. Then there exist \(C, D> 0\) such that for any \(\alpha \geq 0\), any
\(n\)-vertex graph \(G\) with \(K_{k+1}\)-density at most \(\alpha\) has
\(K_{r}\)-density at most
\[\frac{k(k-1)\cdots(k-r+1)}{k^r}(1 + C\alpha^{1/(k+2)} + D/n).\]
\end{lemma}

\begin{proof}
We use a straightforward sampling argument. Let \(G\) be as in the statement of
the lemma. Let \(\zeta {n \choose l}\) be the number of \(l\)-subsets \(U
\subset V(G)\) such that \(G[U]\) contains a copy of \(K_{k+1}\), so that
\(\zeta\) is simply the probability that a uniform random \(l\)-subset of
\(V(G)\) contains a \(K_{k+1}\). Simple counting (or the union bound) gives
\[\zeta \leq {l \choose k+1} \alpha.\]

By Theorem \ref{thm:erd}, each \(K_{k+1}\)-free \(G[U]\) contains
at most
\[{k \choose r}\left(\frac{l}{k}\right)^{r}\]
\(K_{r}\)'s. Therefore, the density of \(K_{r}\)'s in each such \(G[U]\)
satisfies
\begin{align}
\label{eq:upperbd}
d_{K_{r}}(G[U]) & \leq \frac{k(k-1)\cdots(k-r+1)}{k^r}
\frac{l^r}{l(l-1)\cdots(l-r+1)}\nonumber \\
& \leq \frac{k(k-1)\cdots(k-r+1)}{k^r}(1+O(1/l)).
\end{align}
Note that one can choose a random \(r\)-set in graph $G$ by first
choosing a random \(l\)-set $U$, and then choosing a random \(r\)-subset of
$U$. The density of \(K_r\)'s in \(G\) is simply the probability that a uniform
random \(r\)-subset of \(V(G)\) induces a \(K_r\), and therefore
\[d_{K_{r}}(G) = \mathbb{E}_{U} [d_{K_{k}}(G[U])],\]
where the expectation is taken over a uniform random choice of \(U\). If \(U\)
is \(K_{k+1}\)-free, which happens with probability $1-\zeta$, we use the upper
bound (\ref{eq:upperbd}); if \(U\) contains a \(K_{k+1}\), which happens with
probability \(\zeta\), we use the trivial bound \(d_{K_{k}}(G[U]) \leq 1\). We
see that the density of \(K_{r}\)'s in \(G\) satisfies:
\begin{align*}
d_{K_{r}}(G) & \leq (1-\zeta) \frac{k(k-1)\cdots(k-r+1)}{k^r}(1+O(1/l)) +
\zeta\\
& \leq \frac{k(k-1)\cdots(k-r+1)}{k^r}+ O(1/l) + {l \choose k+1} \alpha\\
& \leq \frac{k(k-1)\cdots(k-r+1)}{k^r} + O(1/l) + l^{k+1} \alpha.
\end{align*}
Choosing \(l = \min\{\lfloor \alpha^{-1/(k+2)} \rfloor,n\}\) proves the lemma.
\end{proof}

The second result states that an \(n\)-vertex graph with a small \(K_{k+1}\)-density, a \(K_{k}\)-density not too much less than that of \(T_{k}(n)\), and a \(K_{k-1}\)-density not too much more than that of \(T_{k}(n)\), can be made into a \(k\)-partite graph by the removal of only a small number of edges.

\begin{theorem}
\label{thm:makekpartite}
Let \(G\) be an \(n\)-vertex graph with \(K_{k+1}\)-density at most \(\alpha\),
\(K_{k-1}\)-density at most
\[(1+\beta)\frac{k!}{k^{k-1}},\]
and \(K_{k}\)-density at least
\[(1-\gamma)\frac{k!}{k^{k}},\]
where \(\gamma \leq 1/2\). Then \(G\) can be made into a \(k\)-partite graph
\(G_{0}\) by removing at most
\[\left(2\beta+2\gamma+\frac{8k^{k+1}(k+1)}{k!}\sqrt{\alpha}+2k/n\right){n
\choose 2}\]
edges, which removes at most
\[\left(2\beta+2\gamma+\frac{8k^{k+1}(k+1)}{k!}\sqrt{\alpha}+2k/n\right){k
\choose 2}{n \choose k}\]
\(K_{k}\)'s.
\end{theorem}
\begin{proof}
If \(k \in \mathbb{N}\), and \(G\) is a graph, let
\[\mathcal{K}_{k}(G) = \{S
\in V(G)^{(k)}:\ G[S] \textrm{ is a clique}\}\]
denote the set of all \(k\)-sets
that induce a clique in \(G\). If \(S \subset
V(G)\), let \(N(S)\) denote the set of vertices of \(G\) joined to all vertices
in \(S\), i.e. the intersection of the neighbourhoods of the vertices in \(S\), and let
\(d(S) =
| N(S)|\). For \(S \in \mathcal{K}_{k}(G)\),
let \[f_{G}(S) = \sum_{T \subset S, |T|=k-1} d(T).\]
We begin by sketching the proof. The fact that the ratio
between the \(K_{k}\)-density of \(G\) and the \(K_{k-1}\)-density of \(G\) is
very close to \(1/k\) will imply that the average
\(\mathbb{E}f_{G}(S)\) over all sets \(S \in \mathcal{K}_{k}(G)\) is not too
far below \(n\). The fact that the \(K_{k+1}\)-density of \(G\) is small will
mean that for most sets \(S \in
\mathcal{K}_{k}(G)\), every \((k-1)\)-subset \(T \subset S\) has \(N(T)\)
spanning few edges of \(G\), and any two distinct \((k-1)\)-subsets \(T,T'
\subset S\) have
\(|N(T) \cap N(T')|\) small.
Hence, if we pick such a set \(S\) which has \(f_{G}(S)\) not too far below
the average, the sets \(\{N(T):\ T \subset S, |T|=k-1\}\) will be almost
pairwise disjoint, will cover most of the vertices of
\(G\), and will each span few
edges of \(G\). Small alterations will produce a \(k\)-partition of \(V(G)\)
with few edges of \(G\) within each class, proving the theorem.

We now proceed with the proof. Observe that
\begin{align*}
\mathbb{E}f_{G} & = \frac{\sum_{S \in \mathcal{K}_{k}(G)} \sum_{T \subset S,
| T|=k-1} d(T)}{K_{k}(G)} \\
 & = \frac{\sum_{T \in \mathcal{K}_{k-1}(G)} d(T)^2}{K_{k}(G)}\\
& \geq \frac{\left( \sum_{T \in \mathcal{K}_{k-1}(G)}
d(T)\right)^2}{K_{k-1}(G)K_{k}(G)}\\
& = \frac{(kK_{k}(G))^2}{K_{k-1}(G)K_{k}(G)}\\
& = k^2\frac{K_{k}(G)}{K_{k-1}(G)}\\
& \geq k^2 (1-\gamma) \frac{k!}{k^k} \frac{1}{1+\beta} \frac{k^{k-1}}{k!}
\frac{{n \choose k}}{{n \choose k-1}}\\
& = \frac{1-\gamma}{1+\beta}(n-k+1).
\end{align*}
(The first inequality follows from Cauchy-Schwarz, and the second from our assumptions on the \(K_{k}\)-density and the \(K_{k-1}\)-density of \(G\).)

We call a set \(T \in \mathcal{K}_{k-1}(G)\) \textit{dangerous} if it is contained in at
least \(\sqrt{\alpha} {n-k+1 \choose 2}\) \(K_{k+1}\)'s. Let \(D\) denote the number of dangerous \((k-1)\)-sets. Double-counting the number of times a \((k-1)\)-set is contained in a \(K_{k+1}\), we obtain:
\[D \sqrt{\alpha} {n-k+1 \choose 2} \leq {k+1 \choose 2} \alpha {n \choose k+1},\]
since there are at most \(\alpha {n \choose k+1}\) \(K_{k+1}\)'s in \(G\). Hence,
\[D \leq \sqrt{\alpha} {n \choose k-1}.\]

Similarly, we call a set \(S \in \mathcal{K}_{k}(G)\) \textit{treacherous} if it is contained in
at least \(\sqrt{\alpha} (n-k)\) \(K_{k+1}\)'s. Double-counting the number of times a \(k\)-set is contained in a \(K_{k+1}\), we see that there are at most \(\sqrt{\alpha} {n \choose k}\) treacherous \(k\)-sets.

Call a set \(S \in \mathcal{K}_{k}(G)\) {\em bad} if it is treacherous, or
contains at least one dangerous \((k-1)\)-set; otherwise, call \(S\) {\em good}. Then
the number of bad \(k\)-sets is at most
\[\sqrt{\alpha}{n \choose k}+(n-k+1)\sqrt{\alpha} {n \choose k-1} =
(k+1)\sqrt{\alpha} {n \choose k},\]
so the fraction of sets in \(\mathcal{K}_{k}(G)\) which are bad is at most
\[\frac{(k+1)\sqrt{\alpha}}{(1-\gamma)\frac{k!}{k^{k}}} =
\frac{k^{k}(k+1)\sqrt{\alpha}}{(1-\gamma)k!}.\]
Suppose that
\[\max\{|f_{G}(S)|:\ S \textrm{ is good}\} < (1-\psi)(n-k+1).\]
Observe that for any \(S \in \mathcal{K}_{k}(G)\), we have
\[f_{G}(S) \leq k(n-k+1),\]
since \(d(T) \leq n-k+1\) for each \(T \in S^{(k-1)}\). Hence,
\begin{align*}
\mathbb{E}f_{G} & <
\left(\left(1-\frac{k^{k}(k+1)\sqrt{\alpha}}{(1-\gamma)k!}\right)(1-\psi) +
\frac{k^{k}(k+1)\sqrt{\alpha}}{(1-\gamma)k!}k\right)(n-k+1)\\
 & \leq
\left(1-\psi+\frac{k^{k+1}(k+1)\sqrt{\alpha}}{(1-\gamma)k!}\right)(n-k+1),
\end{align*}
a contradiction if
\[\psi = \psi_{0} := 1-\frac{1-\gamma}{1+\beta} +
\frac{k^{k+1}(k+1)\sqrt{\alpha}}{(1-\gamma)k!} \leq
\gamma+\beta+\frac{2k^{k+1}(k+1)}{k!}\sqrt{\alpha}.\]
Let \(S \in \mathcal{K}_{k}(G)\) be a good \(k\)-set such that \(f_{G}(S) \geq
(1-\psi_{0})(n-k+1)\). Write \(S = \{v_{1},\ldots,v_{k}\}\), let \(T_{i} = S
\setminus \{v_{i}\}\) for each \(i\), and let \(N_i = N(T_i)\) for each \(i\).
Observe that \(N_i \cap N_j = N(S)\) for each \(i \neq j\), and \(|N(S)| = d(S)
\leq \sqrt{\alpha} (n-k)\). Let \(W_{i} = N_{i} \setminus N(S)\) for each
\(i\); observe that the \(W_i\)'s are pairwise disjoint. Let
\[R = V(G)\setminus \cup_{i=1}^{k} W_i\]
be the set of `leftover' vertices.

Observe that
\[\sum_{i=1}^{k}|N_i \setminus N(S)| = f_{G}(S) - kN(S) \geq (1-\psi) (n-k+1) -
k\sqrt{\alpha}(n-k),\]
and therefore the number of leftover vertices satisfies
\[|R| < (\psi +k \sqrt{\alpha})n+k.\]

We now produce a \(k\)-partition \((V_{i})_{i=1}^{k}\) of \(V(G)\) by extending
the partition \((W_i)_{i=1}^{k}\) of \(V(G) \setminus R\) arbitrarily to \(R\),
i.e., we partition the leftover vertices arbitrarily. Now delete all edges of
\(G\) within \(V_i\) for each \(i\). The number of edges within \(N_i\) is
precisely the number of \(K_{k+1}\)'s containing \(T_{i}\), which is at most
\(\sqrt{\alpha} {n -k+1 \choose 2}\). The number of edges incident with \(R\)
is trivially at most \((\psi +k \sqrt{\alpha})n(n-1)+k(n-1)\). Hence, the
number of edges deleted was at most
\begin{align*}
& (\psi +k \sqrt{\alpha})n(n-1)+k(n-1)+k\sqrt{\alpha} {n -k+1 \choose 2}\\
& \leq
\left(2\beta+2\gamma+\frac{8k^{k+1}(k+1)}{k!}\sqrt{\alpha}+2k/n\right){n
\choose 2}.
\end{align*}

Removing an edge removes at most \({n-2 \choose k-2}\) \(K_{k}\)'s, and
therefore the total number of \(K_{k}\)'s removed is at most
\begin{align*}
& \left(2\beta+2\gamma+\frac{8k^{k+1}(k+1)}{k!}\sqrt{\alpha}+2k/n\right){n
\choose 2}{n-2 \choose k-2}\\
&= \left(2\beta+2\gamma+\frac{8k^{k+1}(k+1)}{k!}\sqrt{\alpha}+2k/n\right){k
\choose 2}{n \choose k},
\end{align*}
completing the proof.
\end{proof}

Note that the two results above together imply the following
\begin{corollary}
For any \(k \in \mathbb{N}\), there exist constants \(A_k,B_k >0\) such that the following holds. For any \(\alpha \geq 0\), if \(G\) is an \(n\)-vertex graph with \(K_{k+1}\)-density at most \(\alpha\), and \(K_{k}\)-density at least
\[(1-\gamma)\frac{k!}{k^{k}},\]
where \(\gamma \leq 1/2\), then \(G\) can be made into a \(k\)-partite graph
\(G_{0}\) by removing at most
\[\left(2\gamma+A_{k}\alpha^{1/(k+2)}+B_k / n\right){n
\choose 2}\]
edges, which removes at most
\[\left(2\gamma+A_{k}\alpha^{1/(k+2)}+B_k / n\right){k
\choose 2}{n \choose k}\]
\(K_{k}\)'s.
\end{corollary}

\begin{proof}[Proof of Proposition \ref{prop:kstab}]
Suppose \(\mathcal{G} \subset \mathcal{P}X\) has \(|\mathcal{G}| =m \leq
(1+\eta)|\mathcal{F}_{n,k}|\), and \(k\)-generates at least
\((1-\epsilon)2^{n}\) subsets of \(X\). Our aim is to show that \(\mathcal{G}\) is close to a canonical \(k\)-generator. We may assume that \(\epsilon \leq 1/C_{k}^{k}\) and \(\eta \leq 1/D_{k}^{k}\), so by choosing \(C_k\) and \(D_k\) appropriately large, we may assume throughout that \(\epsilon\) and \(\eta\) are small. By choosing \(\xi_k\) appropriately small, we may assume that \(n \geq n_{0}(k)\), where \(n_{0}(k)\) is any function of \(k\).

We first apply Lemma \ref{lemma:upperbound} and Theorem \ref{thm:makekpartite} with \(G = H[\mathcal{G}]\), where
$H$ is the Kneser graph on $\mathcal{P}X$, \(\mathcal{G} \subset
\mathcal{P}X\) with \(|\mathcal{G}| =m \leq (1+\eta)|\mathcal{F}_{n,k}|\), and
\(\mathcal{G}\) \(k\)-generates at least \((1-\epsilon)2^{n}\) subsets of
\(X\). By (\ref{eq1}),
we have
\[d_{K_{k+1}}(H[\mathcal{G}]) \leq 2^{-a_k n},\]
and therefore we may take \(\alpha = 2^{-a_k n}\). Applying Lemma
\ref{lemma:upperbound} with \(r=k-1\), we may take \(\beta = 2^{-b_{k} n}\) for
some \(b_{k} > 0\).

We have \(|\mathcal{G}|=m \leq (1+\eta)(k2^{n/k}-k)\), so
\[{m \choose k} \leq \frac{m^{k}}{k!} < \frac{(1+\eta)^{k}k^{k}}{k!}2^n.\]
Notice that
\[\sum_{i=0}^{k-1}{m \choose i} \leq km^{k-1} \leq k ((1+\eta)k2^{n/k})^{k-1} <
(1+\eta)^{k-1}k^{k} 2^{(1-1/k)n}.\]
Since \(\mathcal{G}\) \(k\)-generates at least \((1-\epsilon)2^n\)
subsets of \(X\), we have
\[K_{k}(H[\mathcal{G}]) \geq (1-\epsilon)2^{n} -(1+\eta)^{k-1}k^{k}
2^{(1-1/k)n}.\]
Hence,
\begin{align*}
d_{K_{k}}(H[\mathcal{G}]) & = \frac{K_{k}(H[\mathcal{G}])}{{m \choose k}}\\
 & \geq \frac{(1-\epsilon)2^{n} -(1+\eta)^{k-1}k^{k}
2^{(1-1/k)n}}{{(1+\eta)k2^{n/k} \choose k}}\\
& \geq
\frac{1-\epsilon-(1+\eta)^{k-1}k^{k}2^{-n/k}}{(1+\eta)^k}\frac{k!}{k^k}\\
& \geq (1-\epsilon- k\eta - k^{k}2^{-n/k} )\frac{k!}{k^k},
\end{align*}
where the last inequality follows from
$$\frac{1-\epsilon}{(1+\eta)^k} \geq (1-\epsilon)(1-\eta)^k \geq
(1-\epsilon)(1-k\eta) \geq 1-\epsilon-k\eta.$$
Therefore, the $K_{k}$-density of $H[\mathcal{G}]$ satisfies
\[d_{K_{k}}(H[\mathcal{G}]) \geq (1-\gamma)\frac{k!}{k^k},\]
where
\[\gamma = \epsilon + k \eta +k^{k}2^{-n/k}.\]

Let
\[\psi =
\left(2\beta+2\gamma+\frac{8k^{k+1}(k+1)}{k!}\sqrt{\alpha}+2k/n\right){k
\choose 2}.\]
By Theorem \ref{thm:makekpartite}, there exists a \(k\)-partite subgraph
\(G_{0}\) of \(H[\mathcal{G}]\) with
\begin{align*}
K_{k}(G_{0}) & \geq K_{k}(H[\mathcal{G}])-\psi{m \choose k}\\
& \geq (1-\epsilon)2^{n} -(1+\eta)^{k-1}k^{k} 2^{(1-1/k)n} - \psi{m \choose
k}\\
& \geq \left(1-\epsilon-\frac{(1+\eta)^{k}k^{k}}{k!}\psi -
(1+\eta)^{k-1}k^{k}2^{-k/n} \right)2^{n}.
\end{align*}
Writing
\[\phi = \epsilon+\frac{(1+\eta)^{k}k^{k}}{k!}\psi +
(1+\eta)^{k-1}k^{k}2^{-k/n},\]
we have
\[K_{k}(G_{0}) \geq (1-\phi) 2^{n}.\]
Let \(V_{1},\ldots,V_{k}\) be the vertex-classes of \(G_{0}\). By the AM/GM
inequality,
\[K_{k}(G_{0}) \leq \prod_{i=1}^{k}|V_i| \leq
\left(\frac{\sum_{i=1}^{k}|V_i|}{k}\right)^k = (m/k)^{k},\]
and therefore
\begin{equation}
\label{eq2}
| \mathcal{G}|=m \geq k(K_{k}(G_{0}))^{1/k} \geq k(1-\phi)^{1/k}2^{n/k},
\end{equation}
recovering the asymptotic result of \cite{ellis}.

Moreover, any \(k\)-partite graph \(G_{0}\) satisfies
\[e(G_{0}) \geq {k \choose 2}(K_{k}(G_{0}))^{2/k}.\]
To see this, simply apply Shearer's Entropy Lemma with \(S = V(G_0)\), \(\mathcal{F}=\mathcal{K}_{k}(G_0)\), and
\(\mathcal{A} = \{V_i \cup V_j:\ i \neq j\}\). Then \(\mathcal{A}\) is a
\((k-1)\)-cover of \(V(G_0)\). Note that
\(\mathcal{F}_{V_i \cup V_j} \subset E_{G_0}(V_i,V_j)\), and therefore
\[(K_k(G_0))^{k-1} \leq \prod_{\{i,j\} \in [k]^{(2)}}e_{G_{0}}(V_i,V_j).\]
Applying the AM/GM inequality gives:
\[(K_k(G_0))^{k-1} \leq \prod_{\{i,j\}}e_{G_{0}}(V_i,V_j) \leq
\left(\frac{\sum_{\{i,j\}}e_{G_0}(V_i,V_j)}{{k \choose 2}}\right)^{k \choose
2} = \left(\frac{e(G_0)}{{k \choose 2}}\right)^{k \choose 2},\]
and therefore
\[e(G_{0}) \geq {k \choose 2}(K_{k}(G_{0}))^{2/k},\]
as required.

It follows that
\begin{align*}
e(G_{0}) & \geq {k \choose 2} (1-\phi)^{2/k}2^{2n/k}\\
& \geq {k \choose 2}(1-\phi)^{2/k} \left(\frac{m}{(1+\eta)k}\right)^{2}\\
& \geq (1-\eta)^2(1-\phi)^{2/k}(1-1/k) m^2/2\\
& \geq (1-2\eta-\phi^{2/k})(1-1/k) m^2/2\\
& = (1-\delta)(1-1/k) m^2/2,
\end{align*}
where \(\delta = 2\eta+\phi^{2/k}\).

Hence, \(G_0\) is a $k$-partite subgraph of $H[\mathcal{G}]$ with
$|G_0| = |\mathcal{G}|=m$, and $e(G_0) \geq (1-\delta-1/k)m^2/2$. Applying
Lemma \ref{lemma:highmindegree} to \(G_0\), we see that there exists
an induced subgraph \(H'\) of \(G_{0}\) with
\begin{equation}
 \label{eq:H'bound}
|H'| \geq (1-\sqrt{\delta})|\mathcal{G}|,
\end{equation}
and
\[\delta(H') \geq (1-1/k-\sqrt{\delta})(|H'|-1).\]

Let \(Y_{1},\ldots,Y_{k}\) be the vertex-classes
of \(H'\); note that these are families of subsets of \(X\). Clearly, for each
\(i \in [k]\),
\begin{equation}
\label{eq:Yibound}
| Y_i| \leq |H'|-\delta(H') \leq (1/k+\sqrt{\delta})|H'|+1.
\end{equation}
Hence, for each \(i \in [k]\),
\begin{equation}
 \label{eq:Yilowerbound}
| Y_i| \geq |H'| - (k-1)((1/k+\sqrt{\delta})|H'|+1) \geq (1/k-(k-1)
| \sqrt{\delta})|H'| - k+1.
\end{equation}
For each \(i \in [k]\), let
\[S_{i} = \bigcup_{y \in Y_{i}}y\]
be the union of all sets in \(Y_{i}\). We claim that the \(S_{i}\)'s are
pairwise disjoint. Suppose for a contradiction that \(S_{1} \cap S_{2} \neq
\emptyset\). Then there exist \(y_{1} \in Y_{1}\) and \(y_{2}\in Y_{2}\) which
both contain some element \(p \in X\). Since
\[\delta(H') \geq (1-1/k-\sqrt{\delta})(|H'|-1),\]
at least \((1-1/k-\sqrt{\delta})(|H'|-1)\) sets in \(\cup_{i \neq 1}Y_i\) do
not contain \(p\). By (\ref{eq:Yibound}),
\[|\cup_{i \neq 1}Y_i| = \sum_{i \neq 1}|Y_i| \leq
(1-1/k+(k-1)\sqrt{\delta})|H'|+k-1,\]
and therefore the number of sets in \(\cup_{i \neq 1}Y_i\) containing \(p\) is
at most
\[(1-1/k+(k-1)\sqrt{\delta})|H'|+k-1-(1-1/k-\sqrt{\delta})(|H'|-1) \leq
k\sqrt{\delta}|H'|+k.\]
The same holds for the number of sets in \(\cup_{i \neq 2}Y_i\) containing
\(p\), so the total number of sets in \(H'\) containing \(p\) is at most
\[2k\sqrt{\delta}|H'|+2k.\]

Hence, the total number of sets in \(\mathcal{G}\) containing \(p\) is at most
\[(2k+1)\sqrt{\delta}m+2k.\]
But then the number of ways of choosing at most \(k\) disjoint sets in
\(\mathcal{G}\) with one containing \(p\) is at most
\[(1+m^{k-1})((2k+1)\sqrt{\delta}m+2k) = O_{k}(\sqrt{\delta})2^{n} +
O_{k}(2^{(1-1/k)n}) < 2^{n-1}-\epsilon2^{n},\]
contradicting the fact that \(\mathcal{G}\) \(k\)-generates all but \(\epsilon
2^{n}\) of the sets containing \(p\).

Hence, we may conclude that the \(S_{i}\)'s are pairwise disjoint. By
definition, \(Y_{i} \subset \mathcal{P}S_i\), and therefore \(|Y_{i}| \leq
2^{|S_i|}\). But from (\ref{eq:Yilowerbound}),
\begin{align*}
| Y_{i}| & \geq (1-k(k-1) \sqrt{\delta})|H'|/k - k+1\\
& \geq (1-k(k-1) \sqrt{\delta})(1-\sqrt{\delta})|\mathcal{G}|/k - k+1\\
&\geq (1-k(k-1) \sqrt{\delta})(1-\sqrt{\delta})(1-\phi)^{1/k} 2^{n/k} - k+1\\
& \geq (1-(k(k-1)+1)\sqrt{\delta}-\phi^{1/k})2^{n/k}-k+1\\
& > (1-k^2\sqrt{\delta}-\phi^{1/k})2^{n/k}-k\\
& > 2^{n/k-1},
\end{align*}
using (\ref{eq:H'bound}) and (\ref{eq2}) for the second and third inequalities respectively. Hence, we must have \(|S_i| \geq n/k\) for each \(i\), and therefore \(|S_i| = n/k\) for each \(i\), i.e. \((S_i)_{i=1}^{k}\) is an equipartition of \(X\).
Putting everything together and recalling that $\delta=2\eta+\phi^{2/k}$ and
$\phi=O_k(\epsilon +\eta +2^{-c_kn})$, we have
\begin{align*}
|\mathcal{G} \cap (\cup_{i=1}^{k} \mathcal{P}S_i)| & \geq
 \sum_{i=1}^{k}|Y_i|\\
& \geq (1-k^2\sqrt{\delta}-\phi^{1/k})k2^{n/k}-k^2\\
& \geq (1-C_{k}\epsilon^{1/k} - D_{k} \eta^{1/k} - 2^{-\xi_{k}
n})k2^{n/k}
\end{align*}
(provided \(n\) is sufficiently large depending on \(k\)), where \(C_{k},D_{k},\xi_{k}>0\) depend only on \(k\). This proves Proposition
\ref{prop:kstab}.
\end{proof}
We now prove the following

\begin{proposition}
\label{prop:stab}
Let \(\nu(n) = o(1)\). If \(\mathcal{G}\) is a \(k\)-generator for \(X\) with
\(|\mathcal{G}| \leq |\mathcal{F}_{n,k}|\), and
\[|\mathcal{G} \cap \left(\cup_{i=1}^{k} \mathcal{P}S_i\right)| \geq
(1-\nu)|\mathcal{F}_{n,k}|,\]
where \((S_i)_{i=1}^{k}\) is a partition of \(X\) into \(k\) classes of sizes
as equal as possible, then provided \(n\) is
sufficiently large depending on \(k\), we have \(|\mathcal{G}| = |\mathcal{F}_{n,k}|\) and
\[\mathcal{G} = \cup_{i=1}^{k} \mathcal{P}S_i \setminus \{\emptyset\}.\]
\end{proposition}
Note that \(n\) is no longer assumed to be a multiple of \(k\); the case \(k=2\) and \(n\) odd will be needed in Section \ref{sec:2case}.
\begin{proof}
Let \(\mathcal{G}\) and \((S_i)_{i=1}^{k}\) be as in the statement of the
proposition. For each \(i \in [k]\), let
\(\mathcal{F}_{i} = (\mathcal{P}S_{i} \setminus \{\emptyset\}) \setminus
\mathcal{G}\) be the collection of all nonempty subsets of \(S_{i}\) which are
not in \(\mathcal{G}\).
By our assumption on $\mathcal{G}$, we know that \(|\mathcal{F}_{i}| \leq
o(2^{|S_i|})\) for each \(i \in [k]\). Let
\[\mathcal{E} = \mathcal{G} \setminus \bigcup_{i=1}^{k} \mathcal{P}(S_{i})\]
be the collection of `extra' sets in \(\mathcal{G}\); let \(|\mathcal{E}| =
M\).

By relabeling the \(S_{i}\)'s, we may assume that \(|\mathcal{F}_{1}| \geq
| \mathcal{F}_{2}| \geq \cdots \geq |\mathcal{F}_{k}|\). By our assumption on
\(|\mathcal{G}|\), \(M \leq k|\mathcal{F}_{1}|\).

Let
\[\mathcal{R} = \{y_{1} \sqcup s_{2} \sqcup \cdots \sqcup s_{k}:\ y_{1} \in
\mathcal{F}_1,\ s_{i} \subset S_{i}\ \forall i \geq 2\};\]
observe that the sets \(y_{1} \sqcup s_{2} \sqcup \cdots \sqcup s_{k}\) are all
distinct, so \(|\mathcal{R}| = |\mathcal{F}_{1}| 2^{n-|S_1|}\). By considering
the number of sets in \(\mathcal{E}\) needed for \(\mathcal{G}\) to
\(k\)-generate \(\mathcal{R}\), we will show that \(M > k|\mathcal{F}_{1}|\)
unless \(\mathcal{F}_1 = \emptyset\). (In fact, our argument would also show
that \(M > p_k|\mathcal{F}_{1}|\) unless \(\mathcal{F}_1 = \emptyset\), for any
\(p_k >0\) depending only on \(k\).)

Let \(N\) be the number of sets in \(\mathcal{R}\) which may be expressed as a
disjoint union of two sets in \(\mathcal{E}\) and at most \(k-2\) other sets in
\(\mathcal{G}\). Then
\begin{align}
\label{eq:boundforbadsetsk}
N & \leq {M \choose 2} \sum_{i=0}^{k-2}{m \choose i} \nonumber \\
& \leq \tfrac{1}{2} k^{2}|\mathcal{F}_{1}|^{2} (k-1)
\frac{(c_0 k2^{n/k})^{k-2}}{(k-2)!} \nonumber \\
& \leq 4c_0^{k-2}k^k \left(\frac{|\mathcal{F}_{1}|}{2^{|S_1|}}\right)
|\mathcal{F}_1|2^{n-|S_1|}\nonumber \\
& = o(1)|\mathcal{F}_{1}|2^{n-|S_1|}\nonumber \\
& = o(|\mathcal{R}|),
\end{align}
where we have used \(|\mathcal{G}| \leq
|\mathcal{F}_{n,k}| \leq c_0 k 2^{n/k}\) (see (\ref{eq:crudebound})), \(|S_1| \leq \lceil n/k \rceil\), and $|\mathcal{F}_{1}|=o(2^{|S_1|})$ in the second, third and fourth lines respectively.

Now fix \(x_{1} \in \mathcal{F}_{1}\). For \(j \geq 1\), let
\(\mathcal{A}_{j}(x_{1})\) be the collection of \((k-1)\)-tuples
\((s_{2},\ldots,s_{k}) \in \mathcal{P}S_{2} \times \cdots \times
\mathcal{P}S_{k}\) such that
\[x_{1} \sqcup s_{2} \sqcup \cdots \sqcup s_{k}\]
may be expressed as a disjoint union
\[y_{1} \sqcup y_{2} \sqcup \cdots \sqcup y_{k}\]
with \(y_{j} \in \mathcal{E}\) but \(y_{i} \subset S_{i}\ \forall i \neq j\).
Let
\(\mathcal{A}^{*}(x_{1})\) be the collection of \((k-1)\)-tuples
\((s_{2},\ldots,s_{k}) \in \mathcal{P}S_{2} \times \cdots \times
\mathcal{P}S_{k}\) such that
\[x_{1} \sqcup s_{2} \sqcup \cdots \sqcup s_{k}\]
may be expressed as a disjoint union of two sets in \(\mathcal{E}\) and at most
\(k-2\) other sets in \(\mathcal{G}\).

Now fix \(j \neq 1\). For each \((s_{2},\ldots,s_{k}) \in
\mathcal{A}_{j}(x_1)\), we may write
\[x_{1} \sqcup s_{2} \sqcup \cdots \sqcup s_{k} = s_{1}' \sqcup s_{2} \sqcup
\cdots \sqcup s_{j-1} \sqcup y_{j} \sqcup s_{j+1} \sqcup \cdots \sqcup s_{k},\]
where \(y_{j} = s_{j} \sqcup (x_{1} \setminus s_{1}') \in \mathcal{E}\). Since
\(y_{j} \cap S_{j} = s_{j}\), different \(s_{j}\)'s correspond to different
\(y_{j}\)'s \(\in \mathcal{E}\), and so there are at most \(|\mathcal{E}| = M\)
choices for \(s_{j}\). Therefore,
\[|\mathcal{A}_{j}(x_{1})| \leq 2^{n-|S_1|-|S_j|}M \leq
2^{n-|S_1|-|S_j|}k|\mathcal{F}_{1}| \leq 2k
\left(\frac{|\mathcal{F}_1|}{2^{|S_1|}}\right) 2^{n-|S_1|},\]
the last inequality following from the fact that $|S_j| \geq |S_1|-1$. Hence,
\begin{equation}
\label{eq:sumbound}
 \sum_{j = 2}^{k}|\mathcal{A}_{j}(x_{1})| \leq
2k(k-1)\left(\frac{|\mathcal{F}_{1}|}{2^{|S_1|}}\right)2^{n-|S_1|} =
o(1)2^{n-|S_1|}.
\end{equation}

Observe that for each \(x_{1} \in \mathcal{F}_{1}\),
\[\mathcal{A}^{*}(x_{1}) \cup \bigcup_{j=1}^{k} \mathcal{A}_{j}(x_{1}) =
\mathcal{P}S_{2} \times \mathcal{P}S_{3} \times \cdots \times
\mathcal{P}S_{k},\]
and therefore
\[|\mathcal{A}^{*}(x_{1})|+|\mathcal{A}_{1}(x_{1})|+\sum_{j=2}^{k}|\mathcal{A}_{j}(x_{1})|
\geq 2^{n-|S_1|},\]
so by (\ref{eq:sumbound}),
\[|\mathcal{A}^{*}(x_{1})|+|\mathcal{A}_{1}(x_{1})| \geq
(1-o(1))2^{n-|S_1|}.\]

Call \(x_{1} \in \mathcal{F}_{1}\) `bad' if \(|\mathcal{A}^{*}(x_{1})| \geq
2^{-(k+2)} 2^{n-|S_1|}\); otherwise, call \(x_{1}\) `good'. By
(\ref{eq:boundforbadsetsk}), at most a \(o(1)\)-fraction of the sets in
\(\mathcal{F}_{1}\) are bad, so at least a \(1-o(1)\) fraction are good. For
each good set \(x_{1} \in \mathcal{F}_{1}\), notice that
\[|\mathcal{A}_{1}(x_{1})| \geq (1-2^{-(k+2)}-o(1))2^{n-|S_1|}.\]
Now perform the following process. Choose any \((s_{2},\ldots,s_{k}) \in
\mathcal{A}_{1}(x_{1})\); we may write
\[x_{1} \sqcup s_{2} \sqcup \cdots \sqcup s_{k} = z^{(1)} \sqcup s_{2}' \sqcup
\cdots \sqcup s_{k}'\]
with \((s_{2}',\ldots,s_{k}') \in \mathcal{P}S_{2} \times \cdots \times
\mathcal{P}S_{k}\), $ z^{(1)} \in \mathcal{E}$,  \(z^{(1)} \cap S_{1} =
x_{1}\), and \(z^{(1)} \setminus
S_{1} \neq \emptyset\). Pick \(p_{1} \in z^{(1)} \setminus S_{1}\). At most
\(\tfrac{1}{2} 2^{n-|S_1|}\) of the members of \(\mathcal{A}_{1}(x_{1})\) have
union containing \(p_{1}\), so there are at least
\[(1-\tfrac{1}{2}-2^{-(k+2)}-o(1))2^{n-|S_1|}\]
remaining members of \(\mathcal{A}_{1}(x_{1})\). Choose one of these,
\((t_{2},\ldots,t_{k})\) say. By definition, we may write
\[x_{1} \sqcup t_{2} \sqcup \cdots \sqcup t_{k} = z^{(2)} \sqcup t_{2}' \sqcup
\cdots \sqcup t_{k}'\]
with \((t_{2}',\ldots,t_{k}') \in \mathcal{P}S_{2} \times \cdots \times
\mathcal{P}S_{k}\), $ z^{(2)} \in \mathcal{E}$, \(z^{(2)} \cap S_{1} = x_{1}\),
and \(z^{(2)} \setminus
S_{1} \neq \emptyset\). Since \(p_{1} \notin z^{(2)}\), we must have
\(z^{(2)}\neq z^{(1)}\). Pick \(p_{2} \in z^{(2)} \setminus S_{1}\), and
repeat. At most \(\tfrac{3}{4} 2^{n-|S_1|}\) of the members of
\(\mathcal{A}_{1}(x_{1})\) have union containing \(p_{1}\) or \(p_{2}\); there
are at least
\[(\tfrac{1}{4}-2^{-(k+2)}-o(1))2^{n-|S_1|}\]
members remaining. Choose one of these, \((u_{2},\ldots,u_{k})\) say. By
definition, we may write
\[x_{1} \sqcup u_{2} \sqcup \cdots \sqcup u_{k} = z^{(3)} \sqcup u_{2}' \sqcup
\cdots \sqcup u_{k}'\]
with \((u_{2}',\ldots,u_{k}') \in \mathcal{P}S_{2} \times \cdots \times
\mathcal{P}S_{k}\), $ z^{(3)} \in \mathcal{E}$, \(z^{(3)} \cap S_{1} = x_{1}\),
and \(z^{(3)} \setminus
S_{1} \neq \emptyset\). Note that again $z^{(3)}$ is distinct from $z^{(1)},
z^{(2)}$, since $p_1, p_2 \not \in z^{(3)}$.
Continuing this process for \(k+1\) steps, we end up
with a collection of \(k+1\) distinct sets \(z^{(1)},\ldots,z^{(k+1)} \in
\mathcal{E}\) such that \(z^{(l)} \cap S_{1} = x_{1}\ \forall l \in [k+1]\). Do
this for each good set \(x_{1} \in \mathcal{F}_{1}\); the collections produced
are clearly pairwise disjoint. Therefore,
\[|\mathcal{E}| \geq (k+1)(1-o(1))|\mathcal{F}_{1}|.\]
This is a contradiction, unless \(\mathcal{F}_{1} = \emptyset\). Hence, we must have
\(\mathcal{F}_{2} = \cdots = \mathcal{F}_{k} = \emptyset\), and therefore
\[\mathcal{G} = \cup_{i=1}^{k} \mathcal{P}(S_{i}) \setminus \{\emptyset\},\]
proving Proposition \ref{prop:stab}, and completing the proof of Theorem
\ref{thm:kcase}.
\end{proof}

\section{The case \(k=2\) via bipartite subgraphs of \(H\).}
\label{sec:2case}
Our aim in this section is to prove the \(k=2\) case of Conjecture
\ref{conj:frein} for all sufficiently large {\em odd} \(n\), which together
with the \(k=2\) case of Theorem \ref{thm:kcase} will imply

\begin{2caserepeat}
If \(n\) is sufficiently large, \(X\) is an \(n\)-element set, and
\(\mathcal{G} \subset \mathcal{P}X\) is a 2-generator for \(X\), then
\(|\mathcal{G}| \geq |\mathcal{F}_{n,2}|\). Equality holds only if
\(\mathcal{G}\) is of the form \(\mathcal{F}_{n,2}\).
\end{2caserepeat}
Recall that
\[|\mathcal{F}_{n,2}| = \left\{\begin{array}{ll} 2 \cdot 2^{n/2}-2 & \textrm{if }n\textrm{ is even};\\
                                3 \cdot 2^{(n-1)/2}-2 & \textrm{if }n\textrm{ is odd}.\end{array}\right.\]

Suppose that \(X\) is an \(n\)-element set, and \(\mathcal{G} \subset
\mathcal{P}X\) is a 2-generator for \(X\) with \(|\mathcal{G}|=m \leq
| \mathcal{F}_{n,2}|\). The counting argument in the Introduction gives
\[1+m+{m \choose 2} \geq 2^n,\]
which implies that
\[|\mathcal{G}| \geq (1-o(1))\sqrt{2}2^{n/2}.\]
For \(n\) odd, we wish to improve this bound by a factor of approximately \(1.5\).

Our first aim is to prove that induced subgraphs of the Kneser graph \(H\) which have
order
\(\Omega(2^{n/2})\) are \(o(1)\)-close to being bipartite (Proposition
\ref{prop:makebipartite}).

Recall that a graph \(G = (V,E)\) is said to be \(\epsilon\)-{\em close to
being bipartite} if it can be made bipartite by the removal of at most
\(\epsilon |V|^2\) edges, and \(\epsilon\)-{\em far from being bipartite} if it
requires the removal of at least \(\epsilon |V|^2\) edges to make it bipartite.

Using Szemer\'edi's Regularity Lemma, Bollob\'as, Erd\H{o}s, Simonovits and
Szemer\'edi \cite{sz} proved the following.

\begin{theorem}[Bollob\'as, Erd\H{o}s, Simonovits, Szemer\'edi]
\label{thm:sze}
For any \(\epsilon >0\), there exists \(g(\epsilon) \in
\mathbb{N}\) depending on \(\epsilon\) alone such that for {\em any} graph
\(G\) which is \(\epsilon\)-far from being bipartite, the probability that a
uniform random induced subgraph of \(G\) of order \(g(\epsilon)\) is
non-bipartite is at least \(1/2\).
\end{theorem}
Building on methods of Goldreich, Goldwasser and Ron \cite{goldreich}, Alon and
Krivelevich \cite{kriv} proved without using the Regularity Lemma that in fact,
one may take
\begin{equation}
 \label{eq:alonkriv}
g(\epsilon) \leq \frac{(\log(1/\epsilon))^{b}}{\epsilon}
\end{equation}
where \(b>0\) is an absolute constant. As observed in \cite{kriv}, this is
tight up to the poly-logarithmic factor, since necessarily,
\[g(\epsilon) \geq \frac{1}{6 \epsilon}.\]

We will first show that for any fixed \(c >0\) and \(l \in \mathbb{N}\), if
\(\mathcal{A} \subset \mathcal{P}X\) with \(|\mathcal{A}| \geq c2^{n/2}\), then
the density of \(C_{2l+1}\)'s in \(H[\mathcal{A}]\) is at most \(o(1)\). To
prove this, we will show that for any \(l \in \mathbb{N}\), there exists \(t
\in \mathbb{N}\) such that for any fixed \(c > 0\), if \(\mathcal{A} \subset
\mathcal{P}X\) with \(|\mathcal{A}| \geq c2^{n/2}\), then the homomorphism
density of \(C_{2l+1} \otimes t\) in \(H[\mathcal{A}]\) is \(o(1)\). Using
Lemma \ref{lemma:blowupcount}, we will deduce that the homomorphism density of
\(C_{2l+1}\) in \(H[\mathcal{A}]\) is \(o(1)\), implying that the density of
\(C_{2l+1}\)'s in \(H[\mathcal{A}]\) is \(o(1)\). This will show that
\(H[\mathcal{A}]\) is \(o(1)\)-close to being bipartite (Proposition \ref{prop:makebipartite}). To obtain a sharper
estimate for the \(o(1)\) term in Proposition \ref{prop:makebipartite}, we will use
(\ref{eq:alonkriv}), although to prove Theorem \ref{thm:2case}, any \(o(1)\) term would suffice, so one could in fact use Theorem \ref{thm:sze} instead of (\ref{eq:alonkriv}).

We are now ready to prove the following

\begin{proposition}
\label{prop:makebipartite}
Let \(c > 0\). Then there exists \(b >0\) such that for any \(\mathcal{A}
\subset \mathcal{P}X\) with \(|\mathcal{A}| \geq c2^{n/2}\), the induced
subgraph \(H[\mathcal{A}]\) can be made bipartite by removing at most
\[\frac{(\log_{2}\log_{2} n)^b}{\log_2 n}|\mathcal{A}|^2\]
edges.
\end{proposition}

\begin{proof}
Fix \(c > 0\); let \(\mathcal{A} \subset \mathcal{P}X\) with \(|\mathcal{A}| =
m \geq c2^{n/2}\). First, we show that for any fixed \(l \in \mathbb{N}\),
there exists \(t \in \mathbb{N}\) such that the homomorphism density of
\(C_{2l+1} \otimes t\)'s in \(H[\mathcal{A}]\) is at most \(o(1)\). The
argument is a strengthening of that used by Alon and Frankl to prove Lemma 4.2
in \cite{alon}.

Let \(t \in \mathbb{N}\) to be chosen later. Choose \((2l+1)t\) members of
\(\mathcal{A}\) uniformly at random with replacement, \((A_{i}^{(j)})_{1 \leq i
\leq 2l+1,\ 1 \leq j \leq t}\). The homomorphism density of \(C_{2l+1} \otimes
t\) in \(H[\mathcal{A}]\) is precisely the probability that the unions
\[U_{i} = \bigcup_{j=1}^{t}A_{i}^{(j)}\]
satisfy \(U_{i} \cap U_{i+1} = \emptyset\) for each \(i\) (where the addition
is
modulo \(2l+1\)).

We claim that if this occurs, then \(|U_{i}| < (\tfrac{1}{2} - \eta)n\) for
some \(i\), provided \(\eta < 1/(4l+2)\). Suppose for a contradiction that
\(U_{i} \cap U_{i+1} = \emptyset\) for each \(i\), and
\(|U_{i}| \geq (\tfrac{1}{2} - \eta)n\) for each \(i\). Then
$|U_{i+2} \setminus U_{i}| \leq n-|U_{i+1}|-|U_{i}| \leq 2\eta n$
for each \(i \in [2l-1]\). Since $U_{2l+1} \setminus U_{1} \subset \cup_{j=1}^l
(U_{2j+1}\setminus U_{2j-1})$, we have
$|U_{2l+1} \setminus U_{1}| \leq \sum_{j=1}^{l}|U_{2j+1}\setminus U_{2j-1}|
\leq 2l\eta
n$. It follows that \(|U_{1} \cap U_{2l+1}| \geq (1/2 - (2l+1)\eta)n > 0\) if
\(\eta < 1/(4l+2)\), a contradiction.

We now show that the probability of this event is very small. Fix \(i \in
[k]\). Observe that
\begin{align*}
\textrm{Pr}\{|U_{i}| \leq (1/2 - \eta)n\} & = \textrm{Pr}\left(\bigcup_{S
\subset X:|S| \leq (1/2 - \eta)n} \left(\bigcap_{j=1}^{t} \{A_{i}^{(j)} \subset
S\}\right)\right)\\
& \leq \sum_{|S| \leq (1/2 - \eta)n} \textrm{Pr}\left(\bigcap_{j=1}^{t}
\{A_{i}^{(j)} \subset S\}\right)\\
& = \sum_{|S| \leq (1/2 - \eta)n} (2^{|S|}/m)^{t}\\
& \leq 2^{n} \left(\frac{2^{(1/2 - \eta)n}}{c2^{n/2}}\right)^{t} \\
& = 2^{-(\eta t -1)n} c^{-t}\\
& \leq 2^{-n}c^{-t},
\end{align*}
provided \(t \geq 2/\eta\). Hence,
\[\textrm{Pr}\left(\bigcup_{i=1}^{2l+1} \{|U_{i}| \leq (1/2 - \eta)n\}\right)
\leq \sum_{i=1}^{2l+1} \textrm{Pr}\{|U_{i}| \leq (1/2 - \eta)n\} \leq
(2l+1)2^{-n}c^{-t}.\]
Therefore,
\[h_{C_{2l+1} \otimes t}(H[\mathcal{A}]) \leq (2l+1) 2^{-n}c^{-t}.\]
Choose \(\eta = \tfrac{1}{8l}\) and \(t = 2/\eta = 16l\). By Lemma
\ref{lemma:blowupcount},
\begin{align*}
h_{C_{2l+1}}(H[\mathcal{A}]) & \leq ((2l+1) 2^{-n}c^{-t})^{1/t^{2l+1}}\\
& = (2l+1)^{1/(16l)^{2l+1}} 2^{-n/(16l)^{2l+1}}c^{-1/(16l)^{2l}}\\
& = O(2^{-n/(16l)^{2l+1}}).
\end{align*}

Observe that the number of \((2s+1)\)-subsets of \(\mathcal{A}\) containing an
odd cycle of \(H\) is at most
\[\sum_{l=1}^{s} m^{2l+1} h_{C_{2l+1}}(H[\mathcal{A}]) {m-(2l+1) \choose
2(s-l)}.\]
Hence, the probability that a uniform random \((2s+1)\)-subset of
\(\mathcal{A}\) contains an odd cycle of \(H\) is at most
\begin{align*}
\sum_{l=1}^{s} \frac{m^{2l+1}}{m(m-1)\cdots (m-2l)} (2s+1)(2s)\cdots(2(s-l)+1)
h_{C_{2l+1}}(H[\mathcal{A}])\\
\leq s(2s+1)! O(2^{-n/(16s)^{2s+1}}),
\end{align*}
(provided \(s \leq O(\sqrt{m})\)). This can be made \(< 1/2\) by choosing
\[s = a \log_{2}n / \log_{2} \log_2 n,\]
for some suitable \(a > 0\) depending only on \(c\). By (\ref{eq:alonkriv}), it
follows that
\(H[\mathcal{A}]\) is \(((\log_{2}\log_{2} n)^b/\log_2 n)\)-close to being bipartite, for some suitable \(b > 0\) depending only on \(c\),
proving the
proposition.
\end{proof}

Before proving Theorem \ref{thm:2case} for \(n\) odd, we need some more definitions. Let \(X\) be a finite set. If \(A \subset \mathcal{P}X\), and \(i \in X\), we define
\begin{align*}
A_i^{-} & = \{x \in A:\ i \notin x\},\\
A_i^{+} & = \{x \setminus \{i\}:\ x \in A,\ i \in x\};
\end{align*}
these are respectively called the {\em lower} and {\em upper \(i\)-sections} of \(A\).

If \(Y\) and \(Z\) are disjoint subsets of
\(X\), we write \(H[Y,Z]\) for the bipartite subgraph of the Kneser graph \(H\)
consisting of all edges between \(Y\) and \(Z\). If \(B\) is a bipartite
subgraph of \(H\) with vertex-sets \(Y\) and \(Z\), and \(\mathcal{F} \subset
\mathcal{P}X\), we say that \(B\) 2-\textit{generates} \(\mathcal{F}\) if for
every set \(x \in \mathcal{F}\), there exist \(y \in Y\) and \(z \in Z\) such
that \(y \cap z = \emptyset\), \(yz \in E(B)\), and \(y \sqcup z = x\), i.e.
every set in \(\mathcal{F}\) corresponds to an edge of \(B\).

\begin{proof}[Proof of Theorem \ref{thm:2case} for \(n\) odd]
Suppose that \(n=2l+1 \geq 3\) is odd, \(X\) is an \(n\)-element set, and
\(\mathcal{G} \subset \mathcal{P}X\) is a 2-generator for \(X\) with
\(|\mathcal{G}| = m \leq |\mathcal{F}_{n,2}| = 3\cdot 2^{l}-2\). Observe that
\[e(H[\mathcal{G}]) \geq 2^{2l+1}-|\mathcal{G}|-1 \geq 2^{2l+1}-3\cdot 2^l +
1,\]
and therefore \(H[\mathcal{G}]\) has edge-density at least
\[\frac{2^{2l+1}-3\cdot 2^l + 1}{{|\mathcal{G}| \choose 2}} \geq
\frac{2^{2l+1}-3\cdot 2^l + 1}{\tfrac{1}{2}(3\cdot 2^{l}-2)(3 \cdot 2^{l}-3)} >
\tfrac{4}{9}.\]
(Here, the last inequality rearranges to the statement \(l> 0\).) By Proposition \ref{prop:makebipartite} applied to \(\mathcal{G}\), we can remove at
most
\[\frac{(\log_{2}\log_{2} n)^b}{\log_2 n}|\mathcal{G}|^2 <
\frac{(\log_{2}\log_{2} n)^b}{\log_2 n}9\cdot 2^{2l}\]
edges from \(H[\mathcal{G}]\) to produce a bipartite graph \(B\). Let
\(Y,Z\) be the vertex-classes of \(B\); we may assume that \(Y \sqcup Z =
\mathcal{G}\). Define \(\epsilon > 0\) by
\[|\{y \sqcup z:\ y \in Y,\ z \in Z,\ y \cap z = \emptyset\}| = (1-\epsilon)2^{2l+1};\]
then clearly, we have
\begin{equation}\label{eq:upperboundeb}
e(B) \geq (1-\epsilon)2^{2l+1}.
\end{equation}
Note that
\[\epsilon \leq \tfrac{9}{2} \frac{(\log_{2}\log_{2} n)^b}{\log_2 n}+3\cdot
2^{-(l+1)} = O\left(\frac{(\log_{2}\log_{2} n)^b}{\log_2 n}\right) = o(1).\]
Let
\[\alpha = |Y|/2^{l},\ \beta = |Z|/2^{l}.\]
By assumption, \(\alpha+\beta \leq 3-2^{-(l-1)} < 3\). Since \(|Y||Z| \geq e(B)
\geq (2-2\epsilon)2^{2l}\), we have \(\alpha \beta \geq 2-2\epsilon\). This implies that
\begin{equation}
\label{eq:2sidedbound}
1-2\epsilon < \alpha,\beta < 2+2\epsilon.
\end{equation}
(To see this, simply observe that to maximize \(\alpha \beta\) subject to the conditions \(\alpha \leq 1-2\epsilon\) and \(\alpha + \beta \leq 3\), it is best to take \(\alpha = 1-2\epsilon\) and \(\beta = 2+2\epsilon\), giving \(\alpha \beta = 2-2\epsilon-4\epsilon^2 < 2-2\epsilon\), a contradiction. It follows that we must have \(\alpha > 1-2\epsilon\), so \(\beta < 2+2\epsilon\); (\ref{eq:2sidedbound}) follows by symmetry.)

From now on, we think of \(X\) as the set \([n] = \{1,2,\ldots,n\}\). Let
\begin{align*}
W_{1} & = \{i \in [n]:\ |Y_{i}^{+}| \geq |Y|/3\},\\
W_{2} & = \{i \in [n]:\ |Z_{i}^{+}| \geq |Z|/3\}.
\end{align*}
First, we prove the following
\begin{claim}
\label{claim:1}
\(W_{1} \cup W_{2} = [n]\).
\end{claim}
\begin{proof}
Suppose for a contradiction that \(W_{1} \cup W_{2} \neq [n]\). Without loss of generality, we may assume that \(n \notin
W_{1} \cup W_{2}\). Let
\[\theta= |Y_n^{+}|/|Y|,\ \phi = |Z_n^{+}|/|Z|;\]
then we have \(\theta,\phi \leq 1/3\). Observe that the number \(e_n\) of edges between \(Y\) and \(Z\) which generate
a set containing \(n\) satisfies
\begin{equation}
\label{eq:nedgebound}
(1-2\epsilon)2^{2l} \leq e_{n} \leq (\theta \alpha (1-\phi) \beta+ \phi
\beta(1-\theta)\alpha)2^{2l} = (\theta+\phi-2\theta\phi)\alpha \beta 2^{2l}.
\end{equation}
(Here, the left-hand inequality comes from the fact that \(B\) 2-generates
all but at most $\epsilon 2^{2l+1}$ subsets of $[n]$, and therefore $B$
\(2\)-generates
at least \((1-2\epsilon)2^{2l}\) sets containing \(n\).)

Notice that the function
\[f(\theta,\phi) = \theta+\phi-2\theta\phi,\ 0 \leq \theta,\phi \leq 1/3\]
is a strictly increasing function of both \(\theta\) and \(\phi\) for \(0 \leq
\theta,\phi \leq 1/3\), and therefore attains its maximum of \(4/9\) at
\(\theta=\phi=1/3\). Therefore,
\[1-2\epsilon \leq \tfrac{4}{9}\alpha \beta;\]
since \(\alpha + \beta \leq 3\), we have
\[3/2-3\sqrt{\epsilon/2} \leq \alpha,\beta \leq 3/2+3\sqrt{\epsilon/2}.\]
Moreover, by the AM/GM inequality, \(\alpha \beta \leq 9/4\), so
\begin{equation}
 \label{eq:flowerbound}
1-2\epsilon \leq \tfrac{9}{4}f(\theta,\phi),
\end{equation}
and therefore
\[1/3 - 8\epsilon/3 \leq \theta,\phi \leq 1/3.\]
Thus $|Y|, |Z|=(3/2-o(1))2^l$ and $\theta,\phi =1/3-o(1)$. Therefore, we have
\begin{align*}
|Y_n^{+}| & = 2^{l-1}(1-o(1)),\\
|Z_n^{+}| & = 2^{l-1}(1-o(1)),\\
|Y_{n}^-| & = 2^{l}(1+o(1)),\\
|Z_{n}^-| & = 2^{l}(1+o(1)).
\end{align*}
Observe that \(\mathcal{G}_{n}^{-} = Y_{n}^- \cup Z_{n}^{-}\) must \(2\)-generate
all but at most \(o(2^{2l})\) of the sets in \(\mathcal{P}\{1,2,\ldots,n-1\} =
\mathcal{P}\{1,2,\ldots,2l\}\),
and therefore, by Proposition \ref{prop:kstab} for \(k = 2\) and \(n\) even,
there exists an equipartition
\(S_{1} \cup S_{2}\) of \(\{1,2,\ldots,2l\}\) such that \(Y_{n}^{-}\) contains at least
\((1-o(1))2^{l}\) members of
\(\mathcal{P}S_{1}\), and \(Z_{n}^{-}\) contains at least \((1-o(1))2^{l}\) members of
\(\mathcal{P}S_{2}\). Define
\begin{align*}
U & = \{y \in Y:\ y \cap S_2 = \emptyset\},\\
V & = \{z \in Z:\ z \cap S_1 = \emptyset\}.
\end{align*}
Since \(|U_{n}^{-}| = (1-o(1))2^{l}\) and \(|V_{n}^{-}| = (1-o(1))2^{l}\), we must have \(|Y_{n}^{-} \setminus U_{n}^{-}| = o(2^{l})\), and \(|Z_{n}^{-} \setminus V_{n}^{-}| = o(2^{l})\). Our aim is now to show that \(|Y_{n}^{+} \setminus U_{n}^{+}| = o(2^l)\), and \(|Z_{n}^{+} \setminus V_{n}^{+}| = o(2^l)\).

Clearly, we have \(U_{n}^{-} \subset \mathcal{P}S_1\), and \(V_{n}^{-} \subset \mathcal{P}S_2\), so \(|U_{n}^{-}| \leq 2^{l}\) and \(|V_{n}^{-}| \leq 2^{l}\). Moreover, each set \(x \in Y_n^{+} \setminus U_n^{+}\) contains an element of \(S_2\), and therefore \(x \cup \{n\}\) is disjoint from at most \(2^{l-1}\) sets in \(V_{n}^{-} \subset \mathcal{P}S_2\). Similarly, each set \(x \in Z_n^{+} \setminus V_n^{+}\) contains an element of \(S_1\), and therefore \(x \cup \{n\}\) is disjoint from at most \(2^{l-1}\) sets in \(U_{n}^{-} \subset \mathcal{P}S_1\). It follows that
\begin{align*}
e_n \leq & |U_n^{+}| |V_{n}^{-}| + |Y_n^{+} \setminus U_{n}^{+}|2^{l-1} + |V_n^{+}| |U_{n}^{-}| + |Z_n^{+} \setminus V_{n}^{+}|2^{l-1}\\
& + |Y_{n}^{-} \setminus U_n^{-}| |Z_{n}^{+}| + |Z_{n}^{-} \setminus V_n^{-}||Y_{n}^{+}|\\
 \leq & |U_n^{+}| 2^{l} + |Y_n^{+} \setminus U_{n}^{+}|2^{l-1}+ |V_n^{+}| 2^{l} + |Z_n^{+} \setminus V_{n}^{+}|2^{l-1}+o(2^{2l}).
\end{align*}
On the other hand, by (\ref{eq:nedgebound}), we have \(e_n \geq (1-o(1))2^{2l}\). Since \(|Y_n^{+}| = 2^{l-1}(1-o(1))\), and \(|Z_n^{+}| = 2^{l-1}(1-o(1))\), we must have \(|Y_{n}^{+} \setminus U_{n}^{+}| = o(2^l)\), and \(|Z_{n}^{+} \setminus V_{n}^{+}| = o(2^l)\), as required.

We may conclude that \(|Y \setminus U| = o(2^l)\) and \(|Z \setminus V| = o(2^l)\). Hence, there are at most $o(2^l)$ sets in $Y \cup Z=\mathcal{G}$ that
intersect both $S_1$ and $S_2$. On the other hand, since \(|Y_n^{+}| = (1-o(1))2^{l-1}\) and \(|Z_{n}^{+}| = (1-o(1))2^{l-1}\), there are at least $(1+o(1))2^{l-1}$ sets $s_1 \subset S_1$ such that $s_{1}
\cup \{n\} \notin Y$, and there are at least $(1+o(1))2^{l-1}$ sets $s_2
\subset S_2$ such that $s_{2} \cup \{n\} \notin Z$.
Taking all pairs $s_1, s_2$ gives at least \((1+o(1))2^{2l-2}\)
sets of the form
\begin{equation}
\label{eq:form}
\{n\} \cup s_{1} \cup s_{2}\quad (s_1 \subset S_1,\ s_{1} \cup \{n\} \notin Y,\
s_2 \subset S_2,\ s_{2} \cup \{n\} \notin Z).
\end{equation}
Each of these requires a set intersecting both \(S_{1}\) and \(S_{2}\) to
express it
as a disjoint union of two sets from \(\mathcal{G}\).
Since there are \(o(2^{l})\) members of \(\mathcal{G}\) intersecting both
$S_{1}$
and $S_{2}$,
\(\mathcal{G}\) generates at most
\[(|\mathcal{G}|+1)o(2^l) = o(2^{2l})\]
sets of the form (\ref{eq:form}), a contradiction. This proves the claim.
\end{proof}

We now prove the following
\begin{claim}
\(W_{1} \cap W_{2} = \emptyset\).
\end{claim}
\begin{proof}
Suppose for a contradiction that \(W_{1} \cap W_{2}\neq \emptyset\). Without loss of generality, we may assume that \(n
\in W_{1} \cap W_{2}\). As before, let
\[\theta= |Y_{n}^{+}|/|Y|,\ \phi= |Z_{n}^{+}|/|Z|;\]
this time, we have \(\theta,\phi \geq 1/3\). Observe that
\begin{equation}
\label{eq:edgebound}
(2-2\epsilon)2^{2l} \leq e(B) \leq (1-\theta \phi) \alpha \beta 2^{2l}.
\end{equation}
Here, the left-hand inequality is (\ref{eq:upperboundeb}), and the right-hand inequality comes from the fact that there are no edges between pairs of sets \((y,z) \in Y \times Z\) such that \(n \in y \cap z\). Since \(1-\theta \phi \leq 8/9\), we have
\[2-2\epsilon \leq \tfrac{8}{9}\alpha\beta.\]
Since \(\alpha+\beta \leq 3\), it follows that
\[\tfrac{3}{2}(1-\sqrt{\epsilon}) \leq \alpha,\beta \leq
\tfrac{3}{2}(1+\sqrt{\epsilon}).\]
Since \(\alpha \beta \leq 9/4\), we have
\[2-2\epsilon \leq \tfrac{9}{4}(1-\theta\phi),\]
and therefore
\[1/3 \leq \theta,\phi \leq 1/3+8\epsilon/3.\]
Hence, we have
\begin{align*}
|Y_n^{+}| & = 2^{l-1}(1-o(1)),\\
|Z_n^{+}| & = 2^{l-1}(1-o(1)),\\
|Y_{n}^-| & = 2^{l}(1+o(1)),\\
|Z_{n}^-| & = 2^{l}(1+o(1)),
\end{align*}
so exactly as in the proof of Claim \ref{claim:1}, we obtain a contradiction.
\end{proof}

Claims 1 and 2 together imply that \(W_{1} \cup W_{2}\) is a partition of
\(\{1,2,\ldots,n\}=\{1,2,\ldots,2l+1\}\). We will now show that at least a
\((2/3-o(1))\)-fraction of the sets in \(Y\)
are subsets of \(W_{1}\), and similarly at least a \((2/3-o(1))\)-fraction of
the sets in \(Z\) are subsets of \(W_{2}\). Let
$$\sigma = \frac{|Y\setminus \mathcal{P}(W_{1})|}{|Y|},\quad \quad \tau =
\frac{|Z \setminus \mathcal{P}(W_{2})|}{|Z|}\,.$$

Let \(y \in Y\setminus \mathcal{P}W_{1}\), and choose \(i \in y \cap W_{2}\);
since at least \(|Z|/3\) of the sets in \(Z\) contain \(i\), \(y\) has at most
\(2|Z|/3\) neighbours in \(Z\). Hence,
\begin{equation}
\label{eq:edgebound2}
(2-2\epsilon)2^{2l} \leq e(B)\leq (\tfrac{2}{3} \sigma \alpha \beta
+(1-\sigma)\alpha\beta) 2^{2l} = (1-\sigma/3)\alpha\beta 2^{2l} \leq
(1-\sigma/3)\tfrac{9}{4}2^{2l},
\end{equation}
and therefore
\[\sigma \leq 1/3 + 8\epsilon/3,\]
so
\begin{equation}
\label{eq3}
| Y \cap \mathcal{P}(W_{1})| \geq (2/3 - 8\epsilon/3)|Y|.
\end{equation}
Similarly, \(\tau \leq 1/3+8\epsilon/3\), and therefore $|Z \cap
\mathcal{P}(W_{2})| \geq (2/3 - 8\epsilon/3)|Z|$.

If \(|W_{1}| \leq l-1\), then \(|Y \cap \mathcal{P}(W_{1})| \leq 2^{l-1}\), so
\[|Y| \leq \frac{2^{l-1}}{2/3-8\epsilon/3} = \tfrac{3}{4}
\frac{2^l}{1-4\epsilon} < (1-2\epsilon)2^{l},\]
contradicting (\ref{eq:2sidedbound}). Hence, we must have \(|W_{1}| \geq l\).
Similarly, \(|W_{2}| \geq l\), so \(\{|W_{1}|,|W_{2}|\}=\{l,l+1\}\). Without
loss of generality, we may assume that \(|W_{1}| = l\) and \(|W_{2}|=l+1\).

We now observe that
\begin{equation}
\label{eq:Zbound}
| Z| \geq (3/2-6\epsilon)2^l
\end{equation}
To see this, suppose that \(|Z| = (3/2-\eta)2^l\). Since $|Z|+|Y|< 3 \cdot
2^l$, we have \(|Y| \leq (3/2+\eta)2^l\). Recall that any \(y \in Y \setminus
\mathcal{P}W_1\) has at most \(2|Z|/3\)
neighbours in \(Z\). Thus, we have
\begin{align*}
(2-2\epsilon)2^{2l} & \leq e(B)\\
& \leq |Y \cap \mathcal{P}W_1| |Z| + |Y \setminus \mathcal{P}W_1|
\tfrac{2}{3}|Z|\\
& \leq
2^l(\tfrac{3}{2}-\eta)2^l+(\tfrac{1}{2}+\eta)2^{l}\tfrac{2}{3}(\tfrac{3}{2}-\eta)2^l\\
& = (2-\tfrac{1}{3}\eta-\tfrac{2}{3}\eta^2)2^{2l}.
\end{align*}
Therefore $\eta \leq 6\epsilon$, i.e. $|Z| \geq (3/2-6\epsilon)2^{l}$, as
claimed. Since $|Z|+|Y|< 3 \cdot
2^l$, we have
\begin{equation}
\label{eq:Yupperbound}
|Y| \leq (3/2+6\epsilon)2^{l}.
\end{equation}

We now prove the following
\begin{claim}
\label{claim:3}
(a) \(|\mathcal{P}(W_{1}) \setminus Y| \leq 22\epsilon 2^l\);\\
(b) \(|Z \setminus
\mathcal{P}W_2| \leq (\sqrt{\epsilon}+2\epsilon)2^l\).
\end{claim}
\begin{proof}
We prove this by constructing another bipartite subgraph \(B_2\) of \(H\) with
the same number of vertices as \(B\), and comparing \(e(B_2)\) with \(e(B)\).
First, let
\[D = \min\{| \mathcal{P}(W_{2}) \setminus Z|,|Z \setminus
\mathcal{P}W_{2}|\},\]
add \(D\) new members of \(\mathcal{P}(W_{2}) \setminus Z\) to \(Z\), and
delete \(D\) members of \(Z \setminus \mathcal{P}W_{2}\),
producing a new set \(Z'\) and a new bipartite graph \(B_{1} = H[Y,Z']\). Since
\(|Z'|=|Z| \leq (2+2\epsilon)2^{l}\), we have \(|Z' \setminus \mathcal{P}W_{2}|
\leq \epsilon 2^{l+1}\), i.e. \(Z'\) is
almost
contained within \(\mathcal{P}W_{2}\). Notice that every member \(z \in Z
\setminus \mathcal{P}W_{2}\) had at most \(2|Y|/3\) neighbours in \(Y\), and
every new member of \(Z'\) has at least \(|Y \cap
\mathcal{P}(W_{1})| \geq (2/3-8\epsilon/3)|Y|\) neighbours in \(Y\), using
(\ref{eq3}). Hence,
\[e(B)-e(B_1) \leq \tfrac{8\epsilon}{3}|Y|D \leq
\tfrac{8\epsilon}{3}|Y|\tfrac{2}{3}|Z|\leq \tfrac{16 \epsilon}{9}
\tfrac{9}{4}2^{2l} = 4\epsilon 2^{2l},\]
and therefore
\[e(B_1) \geq e(B)-2\epsilon 2^{2l+1} \geq (1-3\epsilon)2^{2l+1}.\]

Second, let
\[C = \min\{| \mathcal{P}W_{1} \setminus Y|,|Y \setminus \mathcal{P}W_{1}|\},\]
add \(C\) new members of \(\mathcal{P}(W_{1}) \setminus Y\) to \(Y\), and
delete
\(C\) members of \(Y \setminus \mathcal{P}W_{1}\), producing a new set \(Y'\)
and a new bipartite graph \(B_{2} = H[Y',Z']\). Since \(|Y| \geq
(1-2\epsilon)2^{l}\), we have \(|Y' \cap \mathcal{P}W_{1}|
\geq (1-2\epsilon)2^{l}\). Since every deleted member of \(Y\) contained an
element of \(W_{2}\), it had at most \((1+2\epsilon)2^{l}\) neighbours in
\(Z'\). (Indeed, such member of $Y$ intersects $2^l$ sets in
$\mathcal{P}W_{2}$, so has at most \(2^l\) neighbours in \(Z' \cap
\mathcal{P}W_2\); there are \(|Z' \setminus
\mathcal{P}W_{2}| \leq \epsilon 2^{l+1}\) other sets in \(Z'\).) On the other
hand, every new member of \(Y'\) is joined to
all of \(Z' \cap \mathcal{P}W_{2}\), which has size at least \(|Z\cap
\mathcal{P}W_{2}| \geq (3/2-8\epsilon)2^{l}\). It follows that
\begin{equation}
\label{eq:lowerboundeB2}
e(B_{2}) \geq e(B_{1})+C(\tfrac{1}{2}-10\epsilon)2^{l} \geq
(1-3\epsilon)2^{2l+1}+C(\tfrac{1}{2}-10\epsilon)2^{l}.
\end{equation}

We now show that \(e(B_2) \leq (1+\epsilon)2^{2l+1}\). If \(|Y'| \geq 2^l\),
then write \(|Y'| = (1+\phi)2^{l}\) where \(\phi \geq
0\); \(Y'\) contains
all of \(\mathcal{P}W_1\), and \(\phi 2^l\) `extra' sets. We have \(|Z'| \leq
(2-\phi)2^l\), and therefore by (\ref{eq:Zbound}), \(\phi \leq 1/2+6\epsilon
<1\). Note that every `extra' set in \(Y'
\setminus \mathcal{P}W_1\) has at most \(2^{l}\) neighbors in
\(\mathcal{P}W_2\), and therefore at most \((1+2\epsilon)2^l\) neighbours in
\(Z'\). Hence,
\[e(B_{2}) \leq 2^l(2-\phi)2^l + \phi 2^l (1+2\epsilon)2^l =
(1+\phi\epsilon)2^{2l+1} \leq (1+\epsilon)2^{2l+1}.\]

If, on the other hand, \(|Y'| \leq 2^{l}\), then since \(|Y'|+|Z'| \leq 3\cdot
2^l\), we have $e(B_{2}) \leq |Y'||Z'| \leq 2^{2l+1}$.
Hence, we always have
\begin{equation}
\label{eq:upperboundeB2}
e(B_2) \leq (1+\epsilon)2^{2l+1}.
\end{equation}
Combining (\ref{eq:lowerboundeB2}) and (\ref{eq:upperboundeB2}), we see that
\[C \leq \frac{8\epsilon}{1/2-10\epsilon}2^l \leq 20\epsilon 2^l,\]
provided \(\epsilon \leq 1/100\).

This implies \textit{(a)}. Indeed, if $|\mathcal{P}W_{1}
\setminus Y| \leq C\leq 20\epsilon 2^l$, then we are done. Otherwise, by the
definition of \(C\), we have $|Y \setminus \mathcal{P}W_{1}| \leq
20\epsilon 2^l$. Recall that by (\ref{eq:2sidedbound}), $|Y| \geq
(1-2\epsilon)2^l$, and therefore
$$|Y \cap \mathcal{P}W_{1}|=|Y|-|Y \setminus \mathcal{P}W_{1}|\geq
(1-2\epsilon)2^l-20\epsilon 2^l=(1-22\epsilon)2^l.$$
Hence,
\begin{equation}
 \label{eq:Ycontains}
|\mathcal{P}(W_1) \setminus Y| \leq 22\epsilon 2^l,
\end{equation}
proving \textit{(a)}.

Since \(e(B) \geq (1-\epsilon)2^{2l+1}\), \(e(B_2) \leq (1+\epsilon)2^{2l+1}\),
and \(e(B_2) \geq e(B_1)\), we have
\begin{equation}
\label{eq:fewedges}
 e(B_1)-e(B) \leq e(B_2)-e(B) \leq (1+\epsilon)2^{2l+1}-(1-\epsilon)2^{2l+1} =
\epsilon2^{2l+2}
\end{equation}
We now use this to show that
\[D = \min\{| \mathcal{P}(W_{2}) \setminus Z|,|Z \setminus \mathcal{P}W_{2}|\}
\leq \sqrt{\epsilon} 2^{l}.\]

Suppose for a contradiction that \(D \geq \sqrt{\epsilon} 2^{l}\); then it is
easy to see that there must exist \(z \in Z \setminus \mathcal{P}W_{2}\) with
at least
\[2|Y|/3-8 \sqrt{\epsilon}2^l\]
neighbours in \(Y\). Indeed, suppose that every \(z \in Z \setminus
\mathcal{P}W_{2}\) has less than
$2|Y|/3-8 \sqrt{\epsilon}2^l$ neighbors in
$Y$. Recall that every new member of $Z'$ has at least $(2/3-8\epsilon)|Y|$
neighbours in $Y$. Hence,
\[e(B_1) - e(B) >  8D(\sqrt{\epsilon}-\epsilon)|Y| \geq 8 \sqrt{\epsilon}
2^{l}(\sqrt{\epsilon}-\epsilon)(1-2\epsilon)2^l \geq \epsilon 2^{2l+1}\]
since \(\epsilon < 1/16\), contradicting (\ref{eq:fewedges}).

Hence, we may choose \(z \in Z \setminus \mathcal{P}W_{2}\) with at least
\[2|Y|/3-8 \sqrt{\epsilon}2^l\]
neighbours in \(Y\). Without loss of generality, we may assume that \(n \in z
\cap W_1\); then none of these neighbours can
contain \(n\). Hence, \(Y\) contains at most
\[|Y|/3+8 \sqrt{\epsilon}2^l\]
sets containing \(n\). But by (\ref{eq:Ycontains}), \(Y\) contains at least
$(1-44\epsilon)2^{l-1}$
of the subsets of \(W_{1}\) that contain \(n\), and therefore \(|Y| \geq
(3/2-o(1))2^l\). By (\ref{eq:Zbound}), it follows that \(|Y| = (3/2-o(1))2^l\)
and \(|Z| = (3/2+o(1))2^l\), so \(Y\) contains
\((1-o(1))2^{l-1}\) sets containing \(n\). Hence, by (\ref{eq:flowerbound}), so
does \(Z\). As in the proof of Claim \ref{claim:1}, we obtain a contradiction.
This implies that
\[D = \min\{|\mathcal{P}(W_{2}) \setminus Z|,|Z \setminus \mathcal{P}W_{2}|\}
\leq
\sqrt{\epsilon} 2^{l},\]
as desired.

This implies \textit{(b)}. Indeed, if \(|Z \setminus \mathcal{P}W_2| \leq \sqrt{\epsilon} 2^{l}\), then we are
done. Otherwise, by the definition of \(D\), \(|\mathcal{P}(W_{2}) \setminus Z|
\leq \sqrt{\epsilon} 2^{l}\), and therefore
\[|Z \cap \mathcal{P}W_2| \geq (2-\sqrt{\epsilon})2^l.\]
Since $|Z| \leq (2+2\epsilon)2^l$, we have
$$|Z \setminus \mathcal{P}W_{2}|=|Z|-|Z\cap \mathcal{P}W_{2}| \leq
(2+2\epsilon)2^l - (2-\sqrt{\epsilon})2^l = (\sqrt{\epsilon}+2\epsilon)2^l,$$
proving \textit{(b)}.
\end{proof}

We conclude by proving the following
\begin{claim}
\[|\mathcal{P}(W_2) \setminus Z| \leq 4\sqrt{\epsilon}2^l.\]
\end{claim}
\begin{proof}
Let
\[\mathcal{F}_{2}=\mathcal{P}(W_{2}) \setminus Z\]
be the collection of sets in \(\mathcal{P}W_{2}\) which are missing from \(Z\),
and let
\[\mathcal{E}_{1} = Y \setminus \mathcal{P}W_{1}\]
be the set of `extra' members of \(Y\).

Since
\(\mathcal{G}\) is a 2-generator for \(X\), we can express all
\(|\mathcal{F}_{2}|2^l\) sets of the form
\[w_{1} \sqcup f_{2}\ (w_{1} \subset W_{1},f_{2} \in \mathcal{F}_{2})\]
as a disjoint union of two sets in \(\mathcal{G}\). All but at most
\(\epsilon 2^{2l+1}\) of these unions correspond to edges of \(B\). Since \(|Z
\setminus
\mathcal{P}W_{2}| \leq (\sqrt{\epsilon}+2\epsilon)2^{l}\), there are at most
\((\sqrt{\epsilon}+2\epsilon)2^{l}|Y|\) edges of \(B\) meeting sets in \(Z
\setminus
\mathcal{P}W_{2}\). Call these edges of \(B\) `bad', and the rest of the edges
of \(B\) `good'. Fix \(f_{2} \in \mathcal{F}_{2}\); we can express all
\(2^{l}\) sets
of the form
\[w_{1} \sqcup f_{2}\ (w_{1} \subset W_{1})\]
as a disjoint union of two sets in \(\mathcal{G}\). If \(w_{1} \sqcup f_{2}\)
is represented by a good edge, then we may write
\[w_{1} \sqcup f_{2} = y_{1} \sqcup w_{2}\]
where \(y_{1} \in \mathcal{E}_{1}\) with \(y_{1} \cap W_{1}=w_{1}\), and \(w_2
\subset W_2\), so for every such \(w_{1}\), there is a different \(y_{1} \in
\mathcal{E}_{1}\).
By (\ref{eq:Yupperbound}), $|Y| \leq (3/2+6\epsilon) 2^l$, and by (\ref{eq:Ycontains}), $|Y \cap
\mathcal{P}W_1| \geq (1-22\epsilon)2^l$, so
\[|\mathcal{E}_{1}| =|Y|-|\mathcal{P}(W_1) \cap Y| \leq (3/2+6\epsilon) 2^l -(1-22
\epsilon)2^l=(1/2+28\epsilon)2^{l}.\]
Thus, for any \(f_{2} \in \mathcal{F}_{2}\), at most \((1/2+28\epsilon)2^{l}\) unions
of the form \(w_{1} \sqcup f_{2}\) correspond to good edges of \(B\). All the
other
unions are generated by bad edges of \(B\) or are not generated by \(B\) at
all, so
\[(1/2-28\epsilon)2^{l}|\mathcal{F}_{2}| \leq
(2\epsilon+\sqrt{\epsilon})2^{l}|Y|+\epsilon2^{2l+1}.\]
Since $|Y| \leq (3/2+6\epsilon) 2^l$ and $\epsilon$ is small,
\(|\mathcal{F}_{2}| \leq 4\sqrt{\epsilon}2^l\), as required.
\end{proof}

We now know that \(Y\) contains all but at most \(o(2^{l})\) of
\(\mathcal{P}W_{1}\), and \(Z\) contains all but at most \(o(2^{l})\) of
\(\mathcal{P}W_{2}\).
Since $|Y|+|Z| <3\cdot 2^l$, we may conclude that \(|Y| = (1-o(1))2^{l}\) and \(|Z|
=
(2-o(1))2^{l}\).
It follows from Proposition \ref{prop:stab} that provided \(n\) is sufficiently
large, we must have \(\mathcal{G} = \mathcal{P}(W_{1}) \cup \mathcal{P}(W_{2})
\setminus \{\emptyset\}\), completing the proof of
Theorem \ref{thm:2case}.
\end{proof}
\section{Conclusion}
We have been unable to prove Conjecture \ref{conj:frein} for \(k \geq 3\) and
{\em all} sufficiently large \(n\). Recall that if \(\mathcal{G}\) is a
\(k\)-generator for an \(n\)-element set \(X\), then
\[|\mathcal{G}| \geq 2^{n/k}.\]
In view of Proposition \ref{prop:makebipartite}, it is natural to ask whether
for any fixed \(k\), all induced subgraphs of the Kneser graph \(H\) with
\(\Omega(2^{n/k})\) vertices can be made \(k\)-partite by removing at most
\(o(2^{2n/k})\) edges. This is false for \(k=3\), however, as the following
example shows. Let \(n\) be a multiple of 6, and take an equipartition of
\([n]\) into 6 sets \(T_{1},\ldots,T_{6}\) of size \(n/6\). Let
\[\mathcal{A} = \bigcup_{\{i,j\} \in [6]^{(2)}} (T_{i} \cup T_{j});\]
then \(|\mathcal{A}| = 15(2^{n/3})\), and \(H[\mathcal{A}]\) contains a
\(2^{n/3}\)-blow-up of the Kneser graph \(K(6,2)\), which has chromatic number
4. It is easy to see that \(H[\mathcal{A}]\) requires the removal of at least
\(2^{2n/3}\) edges to make it tripartite. Hence, a different argument to that
in Section 3 will be required.

We believe Conjecture \ref{conj:frein} to be true for all \(n\) and \(k\), but
it would seem that different techniques will be required to prove this.

David Ellis\\
\texttt{D.Ellis@dpmms.cam.ac.uk}\\
\\
Benny Sudakov\\
\texttt{b.sudakov@math.ucla.edu}
\end{document}